\documentclass[10pt]{article}%
\usepackage{amsfonts}
\usepackage{amsmath}
\usepackage{amssymb}
\usepackage{graphicx}
\usepackage{bbold}
\usepackage[usenames,dvipsnames]{xcolor}
\usepackage{hyperref}
\hypersetup{colorlinks=true, linkcolor=blue, citecolor=green,
filecolor=black, urlcolor=black }
\providecommand{\U}[1]{\protect\rule{.1in}{.1in}}
\newtheorem{theorem}{Theorem}

\newtheorem{definition}[theorem]{Definition}
\newtheorem{example}[theorem]{Example}

\newtheorem{lemma}[theorem]{Lemma}

\newtheorem{proposition}[theorem]{Proposition}
\newtheorem{remark}[theorem]{Remark}

\newenvironment{proof}[1][Proof]{\noindent\textbf{#1.} }{\ \rule{0.5em}{0.5em}}

\renewcommand{\thefootnote}{\fnsymbol{footnote}}
\begin{document}

\title{Multivalued Stochastic Delay Differential Equations and Related Stochastic
Control Problems}
\author{Bakarime Diomande$^{a,1}$, Lucian Maticiuc$^{a,b,2,\ast}$\bigskip\\{\small $^{a}$ Faculty of Mathematics, \textquotedblleft Alexandru Ioan
Cuza\textquotedblright\ University,}\\{\small Carol I Blvd., no. 11, Ia\c{s}i, 700506, Romania}\\{\small $^{b}$ Department of Mathematics, \textquotedblleft Gheorghe
Asachi\textquotedblright\ Technical University,}\\{\small Carol I Blvd., no. 11, Ia\c{s}i, 700506, Romania}}
\maketitle

\begin{abstract}
We study the existence and uniqueness of a solution for the multivalued
stochastic differential equation with delay (the multivalued term is of
subdifferential type):%
\[
\left\{
\begin{array}
[c]{l}%
dX(t)+\partial\varphi\left(  X(t)\right)  dt\ni b\left(
t,X(t),Y(t),Z(t)\right)  dt\medskip\\
\quad\quad\quad\quad\quad\quad\quad\quad\quad\quad+\sigma\left(
t,X(t),Y(t),Z(t)\right)  dW(t),~t\in(s,T],\medskip\\
X(t)=\xi\left(  t-s\right)  ,\;t\in\left[  s-\delta,s\right]  .
\end{array}
\right.
\]

Specify that in this case the coefficients at time $t$ depends also on
previous values of $X\left(  t\right)  $ through $Y(t)$ and $Z(t)$. Also $X$
is constrained with the help of a bounded variation feedback law $K$ to stay
in the convex set $\overline{\mathrm{Dom}\left(  \varphi\right)  }$.

Afterwards we consider optimal control problems where the state $X$ is a
solution of a controlled delay stochastic system as above. We establish the
dynamic programming principle for the value function and finally we prove that
the value function is a viscosity solution for a suitable
Hamilton-Jacobi-Bellman type equation.

\end{abstract}

\footnotetext[1]{{\scriptsize Corresponding author.}}

\renewcommand{\thefootnote}{\arabic{footnote}} \footnotetext[1]%
{{\scriptsize The work of this author was supported by the project
\textquotedblleft Deterministic and stochastic systems with state
constraints\textquotedblright, code 241/05.10.2011.}} \footnotetext[2]%
{{\scriptsize The work of this author was supported by POSDRU/89/1.5/S/49944
project.}} \renewcommand{\thefootnote}{\fnsymbol{footnote}}
\footnotetext{\textit{{\scriptsize E-mail addresses:}}
{\scriptsize bakarime.diomande@yahoo.com (Bakarime Diomande),
lucian.maticiuc@ymail.com (Lucian\ Maticiuc).}}

\textbf{AMS Classification subjects}: 60H10, 93E20, 49L20, 49L25.$\smallskip$

\textbf{Keywords or phrases}: Multivalued SDE with delay; Dynamic programming
principle; Hamilton-Jacobi-Bellman equation; Viscosity solution.

\section{Introduction}

Stochastic (or deterministic) dynamical systems with delay appear in various
applications where the dynamics are subject to propagation delay. Moreover,
there is a natural motivation in considering the problem of delayed stochastic
differential equations with constraints on the state. Our study concerns first
the existence and uniqueness of a solution for the following stochastic delay
differential equation of multivalued type, also called stochastic delay
variational inequality (where the solution is forced, due to the presence of
term $\partial\varphi\left(  X(t)\right)  $, to remains into the convex set
$\overline{\mathrm{Dom}\left(  \varphi\right)  }\,$):%
\begin{equation}
\left\{
\begin{array}
[c]{r}%
dX(t)+\partial\varphi\left(  X(t)\right)  dt\ni b\left(
t,X(t),Y(t),Z(t)\right)  dt+\sigma\left(  t,X(t),Y(t),Z(t)\right)
dW(t),~t\in(s,T],\medskip\\
\multicolumn{1}{l}{X(t)=\xi\left(  t-s\right)  ,\;t\in\left[  s-\delta
,s\right]  ,}%
\end{array}
\right.  \label{SVI delayed}%
\end{equation}
where $\left(  s,\xi\right)  \in\lbrack0,T)\times\mathcal{C}\big(\left[
-\delta,0\right]  ;\overline{\mathrm{Dom}\left(  \varphi\right)  }\,\big)$ is
arbitrary fixed, $b$ and $\sigma$ are given functions, $\delta\geq0$ is the
fixed delay,%
\begin{equation}
Y(t):=\int_{-\delta}^{0}e^{\lambda r}X(t+r)dr,\quad Z(t):=X(t-\delta)
\label{def of Y,Z}%
\end{equation}
and $\partial\varphi$ is the subdifferential operator associated to $\varphi$.

We mention, as an example, the particular case of $\varphi$ being the
indicator function $I_{\bar{D}}:\mathbb{R}^{d}\rightarrow(-\infty,+\infty]$ of
a nonempty closed convex set $\bar{D}\subset\mathbb{R}^{d}$, i.e. $I_{\bar{D}%
}\left(  x\right)  =0$ if $x\in\bar{D}$ and $+\infty$ if $x\notin\bar{D}$. The
subdifferential is given by%
\[
\partial{I}_{\bar{D}}(x)=\left\{
\begin{array}
[c]{ll}%
0, & \text{if }x\in\mathrm{Int}\left(  D\right)  ,\medskip\\
\mathcal{N}_{\bar{D}}(x), & \text{if }x\in\mathrm{Bd}\left(  D\right)
,\medskip\\
\emptyset, & \text{if }x\notin\bar{D},
\end{array}
\right.
\]
where $\mathcal{N}_{\bar{D}}(x)$ denotes the closed external cone normal to
$\bar{D}$ at $x\in\mathrm{Bd}\left(  D\right)  $.

In this case, the supplementary drift $-\partial{I}_{\bar{D}}(X\left(
t\right)  )$ is an \textquotedblleft inward push\textquotedblright\ that
forbids the process $X\left(  t\right)  $ to leave the domain $\bar{D}$ and
this drift acts only when $X\left(  t\right)  $ reach the boundary of $\bar
{D}$. In the case of $\bar{D}$ being the closed positive orthant from
$\mathbb{R}^{d}$ we recall article Kinnally \& Williams \cite{ki-wi/10} (see
also the reference therein for a more complete literature scene of applications).

The next problem is to minimize the cost functional%
\begin{equation}
J(s,\xi;u)=\mathbb{E}\big[\int_{s}^{T}f\left(  t,X(t),Y(t),u(t)\right)
dt+h\left(  X(T),Y(T)\right)  \big] \label{func cost}%
\end{equation}
over a class of control strategies denoted by $\mathcal{U}\left[  s,T\right]
$ (here $f$ and $h$ are only continuous and with polynomial growth).

We define the value function%
\begin{equation}%
\begin{array}
[c]{l}%
V\left(  s,\xi\right)  =\inf_{u\in\mathcal{U}\left[  s,T\right]  }%
J(s,\xi;u),\;\left(  s,\xi\right)  \in\lbrack0,T)\times\mathcal{C}\big(\left[
-\delta,0\right]  ;\overline{\mathrm{Dom}\left(  \varphi\right)
}\,\big)\medskip\\
V(T,\xi)=h(X(0),Y(0)),\;\xi\in\mathcal{C}\big(\left[  -\delta,0\right]
;\overline{\mathrm{Dom}\left(  \varphi\right)  }\,\big)
\end{array}
\label{funct value}%
\end{equation}
and the second aim will be to prove that $V$ satisfies the dynamic programming
principle and, under the assumption $V\left(  s,\xi\right)  =V\left(
s,x,y\right)  $, the value function is a viscosity solution for
Hamilton-Jacobi-Bellman equation:%
\[
\left\{
\begin{array}
[c]{r}%
\displaystyle-\frac{\partial V}{\partial s}(s,x,y)+\sup_{u\in\mathrm{U}%
}\mathcal{H}\big(s,x,y,z,u,-D_{x}V\left(  s,x,y\right)  ,-D_{xx}^{2}V\left(
s,x,y\right)  \big)\medskip\\
-\langle x-e^{-\lambda\delta}z-\lambda y,D_{y}V\left(  s,x,y\right)
\rangle\in\langle-D_{x}V\left(  s,x,y\right)  ,\partial\varphi\left(
x\right)  \rangle,\medskip\\
\text{for }\left(  s,x,y,z\right)  \in\left(  0,T\right)  \times
\overline{\mathrm{Dom}\left(  \varphi\right)  }\times\mathbb{R}^{2d}%
,\medskip\\
\multicolumn{1}{l}{V\left(  T,x,y\right)  =h\left(  x,y\right)  \text{ for
}\left(  x,y\right)  \in\overline{\mathrm{Dom}\left(  \varphi\right)  }%
\times\mathbb{R}^{d},}%
\end{array}
\right.
\]
where $\mathcal{H}:\left[  0,T\right]  \times\mathbb{R}^{3d}\times
\mathrm{U}\times\mathbb{R}^{d}\times\mathbb{R}^{d\times d}\rightarrow
\mathbb{R}$ is defined by%
\[
\mathcal{H}\left(  s,x,y,z,u,q,X\right)  :=\left\langle b\left(
s,x,y,z,u\right)  ,q\right\rangle +\frac{1}{2}\mathrm{Tr}\left(  \sigma
\sigma^{\ast}\right)  \left(  s,x,y,z,u\right)  X-f\left(  s,x,y,u\right)  .
\]
We recall that the existence problem for stochastic equation
(\ref{SVI delayed}) without the multivalued term $\partial\varphi$ has been
treated by Mohammed in \cite{mo/98} (see also \cite{mo/84}). On the other
hand, the variational inequality%
\begin{equation}
\left\{
\begin{array}
[c]{r}%
dX(t)+\partial\varphi\left(  X(t)\right)  dt\ni b\left(  t,X(t)\right)
dt+\sigma\left(  t,X(t)\right)  dW(t),~t\in(s,T],\medskip\\
\multicolumn{1}{l}{X(s)=\xi,}%
\end{array}
\right.  \label{SVI not-delayed 2}%
\end{equation}
has been considered in Bensoussan \& R\u{a}scanu \cite{be-ra/94} (for the
first time) and in Asiminoaei \& R\u{a}\c{s}canu \cite{as-ra/97} (where the
existence is proved through a penalized method). After that the results are
extended in R\u{a}\c{s}canu \cite{ra/96} (the Hilbert space framework) and in
C\'{e}pa \cite{ce/98} (the finite dimensional case) by considering a maximal
monotone operator $A$ instead of $\partial\varphi$:%
\begin{equation}
\left\{
\begin{array}
[c]{r}%
dX(t)+A\left(  X(t)\right)  dt\ni b\left(  t,X(t)\right)  dt+\sigma\left(
t,X(t)\right)  dW(t),~t\in(s,T],\medskip\\
\multicolumn{1}{l}{X(s)=\xi.}%
\end{array}
\right.  \label{SVI not-delayed}%
\end{equation}
More recently, in \cite{bu-ma-ra/13}, the existence results for
(\ref{SVI not-delayed 2}) have been extended to the non-convex domains case by
considering the Fr\'{e}chet subdifferential $\partial^{-}\varphi$ in the place
of $\partial\varphi$.

Stochastic optimal control subject to multivalued stochastic equation
(\ref{SVI not-delayed}) has been treated in Z\u{a}linescu \cite{za/08} where
it is prove first the existence of a weak solution for equation of type
(\ref{SVI not-delayed}) and after the existence of an optimal relaxed control.
In Zalinescu \cite{za/12}, the author consider the controlled equation
(\ref{SVI not-delayed}) and the cost functional to minimize given by
$J(s,x;u)=\mathbb{E}\big[\int_{s}^{T}f\left(  t,X(t),u(t)\right)  dt+h\left(
X(T)\right)  \big].$ In order to prove the dynamic programming principle and
the viscosity property of the value function, the Yosida approximation of
operator $A$ and the optimal control problem for the penalized equation were considered.

In the case of a controlled system of type (\ref{SVI delayed}) we mention the
recent work \cite{di-za/13} where are establish sufficient and necessary
conditions of the maximum principle; in the case of $\varphi$ being zero, we
refer to the paper Larssen \cite{la/02} where it is establish, under Lipschitz
assumptions of the coefficients $f$ and $h$, that the value function satisfies
the dynamic programming principle. This work allowed Larssen \& Risebro in
\cite{la-ri/03} to prove, in the frame of the delay systems and under some
supplementary assumption on $V$, that the value function is viscosity solution
for a Hamilton-Jacobi-Bellman equation.

Concerning the problem of finding value function (associated to a system with
delay), it is very difficult to treat it since the space of initial data is
infinite dimensional. Nonetheless it happens that choosing a specific
structure of the dependence of the past and under certain conditions the
control problem for systems with delay can be reduce to a finite dimensional
problem and some results have been obtained (see, e.g., Kolmanovskii \&
Shaikhet \cite{ko-sh/96}). In the same framework, Elsanousi \& Larssen in
\cite{el-la/01} have studied a delayed linear system with $f$ and $h$ of HARA
utility type and Elsanousi, Oksendal \& Sulem \cite{el-ok-su/00} have
considered a singular stochastic control problem for a certain linear delay
system. Larssen \& Risebro in \cite{la-ri/03} are seeking for conditions that
ensure to a solution of Hamilton-Jacobi-Bellman equation to be independent of
$z=z\left(  \xi\right)  :=\xi\left(  -\delta\right)  $ and to depend only on
$x=x\left(  \xi\right)  :=\xi\left(  0\right)  $ and $y=y\left(  \xi\right)
:=\int_{-\delta}^{0}e^{\lambda r}\xi\left(  r\right)  dr$.

Therefore, in order to show that $V$ given by (\ref{funct value}) is viscosity
solution of a Hamilton-Jacobi-Bellman equation, the assumption that the value
function $V$ depend on $\xi$ only through $x$ and $y$ occurs naturally.

The paper is organized as follows: in section 2 we introduce the stochastic
delay variational inequalities and we provide the notations and the
assumptions used throughout the paper; some a priori estimates of the solution
are also given. The last part of this section is devoted to the proof of the
existence theorem. Section 3 is dedicated to the optimal control problem: we
first show that the value function satisfies the dynamic programming
principle, then it is proved that the value function is a viscosity solution
of a proper Hamilton-Jacobi-Bellman equation.

\section{Stochastic variational inequalities with delay}

\subsection{Notations and assumptions}

Let $s\in\lbrack0,T)$ be arbitrary but fixed and $\left(  \Omega
,\mathcal{F},\{\mathcal{F}_{t}^{s}\}_{t\geq s},\mathbb{P}\right)  $ be a
complete probability space. The process $\left\{  W\left(  t\right)  \right\}
_{t\geq s}$ is a $n$-dimensional standard Brownian motion with $W\left(
s\right)  =0$ and we suppose that $\{\mathcal{F}_{t}^{s}\}_{t\geq s}$ is a filtration.

The notation $\mathrm{BV}\left(  \left[  a,b\right]  ;\mathbb{R}^{d}\right)  $
stands for the space of bounded variation functions defined on $\left[
a,b\right]  $ and $\mathcal{C}([a,b];\mathbb{R}^{d})$ for the space of
continuous functions on $[a,b]$ endowed with the supremum norm:%
\[
\left\Vert X\right\Vert _{\mathcal{C}([a,b];\mathbb{R}^{d})}=\sup_{s\in\lbrack
a,b]}\left\vert X(s)\right\vert .
\]
If $K\in\mathrm{BV}\left(  \left[  a,b\right]  ;\mathbb{R}^{d}\right)  $ then
$\left\Vert K\right\Vert _{\mathrm{BV}\left(  \left[  a,b\right]
;\mathbb{R}^{d}\right)  }$ will denote its variation on $\left[  a,b\right]  $.

For $1\leq p<+\infty$ let $L_{ad}^{p}\left(  \Omega;\mathcal{C}\left(  \left[
a,b\right]  ;\mathbb{R}^{d}\right)  \right)  $ be the closed linear subspace
of stochastic processes $X\in L^{p}\left(  \Omega;\mathcal{C}\left(  \left[
a,b\right]  ;\mathbb{R}^{d}\right)  \right)  $ which are $\mathcal{F}_{t}^{s}$-adapted.

The equation envisaged is%
\begin{equation}
\left\{
\begin{array}
[c]{r}%
dX(t)+\partial\varphi\left(  X(t)\right)  dt\ni b\left(
t,X(t),Y(t),Z(t)\right)  dt+\sigma\left(  t,X(t),Y(t),Z(t)\right)
dW(t),\medskip\\
t\in(s,T],\medskip\\
\multicolumn{1}{l}{X(t)=\xi\left(  t-s\right)  ,\;t\in\left[  s-\delta
,s\right]  ,}%
\end{array}
\right.  \label{SVI delayed 2}%
\end{equation}
where $\delta\geq0$ is a fixed delay and $\xi\in\mathcal{C}\big(\left[
-\delta,0\right]  ;\overline{\mathrm{Dom}\left(  \varphi\right)  }\,\big)$ is
arbitrary fixed. The functions $Y$ and $Z$ are defined by (\ref{def of Y,Z}).

We will need the following assumptions:

\begin{itemize}
\item[$\mathrm{(H}_{1}\mathrm{)}$] The function $\varphi:\mathbb{R}%
^{d}\rightarrow(-\infty,+\infty]$ is convex and lower semicontinuous (l.s.c.)
such that%
\[
\mathrm{Int}\left(  \mathrm{Dom}\left(  \varphi\right)  \right)  \neq
\emptyset,
\]
where $\mathrm{Dom}\left(  \varphi\right)  :=\{x\in\mathbb{R}^{d}%
:\varphi\left(  x\right)  <+\infty\}$ and suppose that$^{\ast}$%
\footnotetext[1]{$\;${\scriptsize In fact this assumption is not a restriciton
since we can take $x_{0}\in\mathrm{Int}\left(  \mathrm{Dom}\left(
\varphi\right)  \right)  $, $x_{0}^{\ast}\in\partial\varphi\left(
x_{0}\right)  $ and we can replace $\varphi\left(  x\right)  $ by
$\varphi\left(  x+x_{0}\right)  -\varphi\left(  x_{0}\right)  -\left\langle
x_{0}^{\ast},x_{0}\right\rangle .$}}%
\[
0\in\mathrm{Int}\left(  \mathrm{Dom}\left(  \varphi\right)  \right)
\quad\text{and}\quad\varphi(x)\geq\varphi\left(  0\right)  =0,\;\forall
\,x\in\mathbb{R}^{d};
\]

\end{itemize}

We recall that the subdifferential of the function $\varphi$ is defined by%
\[
\partial\varphi(x)=\left\{  y\in\mathbb{R}^{d}:\left\langle y,z-x\right\rangle
+\varphi(x)\leq\varphi(z),\;\forall z\in\mathbb{R}^{d}\right\}
\]
and by $\left(  x,x^{\ast}\right)  \in\partial\varphi$ we will understand that
$x\in\mathrm{Dom}(\partial\varphi)$ and $x^{\ast}\in\partial\varphi(x)$, where%
\[
\mathrm{Dom}(\partial\varphi):=\left\{  x\in\mathbb{R}^{d}:\partial
\varphi(x)\neq\emptyset\right\}  .
\]

\begin{example}
A particular case of $\partial\varphi$ is obtained by considering a nonempty
closed convex subset $\bar{D}$ of $\mathbb{R}^{d}$ and the indicator function
$I_{\bar{D}}:\mathbb{R}^{d}\rightarrow(-\infty,+\infty]$, i.e.%
\[
I_{\bar{D}}\left(  x\right)  :=\left\{
\begin{array}
[c]{rl}%
0, & \text{if }x\in\bar{D},\medskip\\
+\infty, & \text{if }x\in\mathbb{R}^{d}\setminus\bar{D},
\end{array}
\right.
\]
which is a proper convex lower semicontinuous.

In this case the subdifferential of $I_{\bar{D}}$ becomes:%
\[
\partial{I}_{\bar{D}}(x)=\left\{
\begin{array}
[c]{ll}%
0, & \text{if }x\in\mathrm{Int}\left(  D\right)  ,\medskip\\
\mathcal{N}_{\bar{D}}(x)=\{\langle y,z-x\rangle\leq0,~z\in D\}, & \text{if
}x\in\mathrm{Bd}\left(  D\right)  ,\medskip\\
\emptyset, & \text{if }x\notin\bar{D},
\end{array}
\right.
\]
where $\mathcal{N}_{\bar{D}}(x)$ denotes the closed external cone normal to
$\bar{D}$ at $x\in\mathrm{Bd}\left(  D\right)  $.
\end{example}

The existence of a solution for (\ref{SVI delayed 2}) will be shown using the
penalized problem. More precisely we considered the Yosida approximation of
the operator $\partial\varphi$: for $\epsilon\in(0,1]$ let $\nabla
\varphi_{\epsilon}$ be the gradient of $\varphi_{\epsilon}$, where
$\varphi_{\epsilon}$ is the Moreau-Yosida regularization of $\varphi$, i.e.%
\begin{equation}
\varphi_{\epsilon}(x):=\inf\{\frac{1}{2\epsilon}|v-x|^{2}+\varphi
(v):v\in\mathbb{R}^{d}\},\;\epsilon>0, \label{Yosida regular}%
\end{equation}
which is a $C^{1}$ convex function.

We recall some useful inequalities (since $\varphi$ satisfies assumption
$\mathrm{(H}_{1}\mathrm{)}$): for all $x,y\in\mathbb{R}^{d}$%
\begin{equation}%
\begin{array}
[c]{rl}%
\left(  i\right)  & \varphi_{\epsilon}\left(  x\right)  =\dfrac{\epsilon}%
{2}\left\vert \nabla\varphi_{\epsilon}(x)\right\vert ^{2}+\varphi\left(
x-\epsilon\nabla\varphi_{\epsilon}(x)\right)  ,\medskip\\
\left(  ii\right)  & \varphi\left(  J_{\epsilon}\left(  x\right)  \right)
\leq\varphi_{\epsilon}\left(  x\right)  \leq\varphi\left(  x\right)
,\medskip\\
\left(  iii\right)  & \nabla\varphi_{\epsilon}\left(  x\right)  =\partial
\varphi_{\epsilon}\left(  x\right)  \in\partial\varphi\left(  J_{\epsilon
}\left(  x\right)  \right)  ,\medskip\\
\left(  iv\right)  & \left\vert \nabla\varphi_{\epsilon}(x)-\nabla
\varphi_{\epsilon}(y)\right\vert \leq\dfrac{1}{\epsilon}\left\vert
x-y\right\vert ,\medskip\\
\left(  v\right)  & \left\langle \nabla\varphi_{\epsilon}(x)-\nabla
\varphi_{\epsilon}(y),x-y\right\rangle \geq0,\medskip\\
\left(  vi\right)  & \left\langle \nabla\varphi_{\epsilon}(x)-\nabla
\varphi_{\delta}(y),x-y\right\rangle \geq-(\epsilon+\delta)\left\langle
\nabla\varphi_{\epsilon}(x),\nabla\varphi_{\delta}(y)\right\rangle ,\medskip\\
\left(  vii\right)  & \forall u_{0}\in\mathrm{Int}\left(  \mathrm{Dom}\left(
\varphi\right)  \right)  ,\;\exists r_{0}>0,\exists M_{0}>0\text{ such
that}\medskip\\
& \quad r_{0}\left\vert \nabla\varphi_{\epsilon}\left(  x\right)  \right\vert
\leq\left\langle \nabla\varphi_{\epsilon}\left(  x\right)  ,x-u_{0}%
\right\rangle +M_{0},\;\forall\epsilon>0,\;\forall x\in\mathbb{R}^{d},
\end{array}
\label{ineq Yosida}%
\end{equation}
where $J_{\epsilon}\left(  x\right)  :=x-\epsilon\nabla\varphi_{\epsilon}(x)$
(for the proof see \cite{ba/10} and \cite{as-ra/97} for the last one).

Moreover, since $\varphi(x)\geq\varphi\left(  0\right)  =0,\;\forall
\,x\in\mathbb{R}^{d}$,%
\begin{equation}%
\begin{array}
[c]{rl}%
\left(  viii\right)  & \left\vert \nabla\varphi_{\epsilon}\left(  x\right)
\right\vert \leq\dfrac{1}{\epsilon}\left\vert x\right\vert \medskip\\
\left(  ix\right)  & \dfrac{\epsilon}{2}\left\vert \nabla\varphi_{\epsilon
}\left(  x\right)  \right\vert ^{2}\leq\varphi_{\epsilon}\left(  x\right)
\leq\left\langle \nabla\varphi_{\epsilon}\left(  x\right)  ,x\right\rangle .
\end{array}
\label{ineq Yosida 2}%
\end{equation}

\begin{remark}
Under assumption $\mathrm{(H}_{1}\mathrm{)}$ the subdifferential operator
$\partial\varphi$ becomes a maximal monotone ope{\footnotesize \-}%
ra{\footnotesize \-}tor, i.e. maximal in the class of operators which satisfy
the condition%
\[
\left\langle y^{\ast}-z^{\ast},y-z\right\rangle \geq0~,\;\forall\left(
y,y^{\ast}\right)  ,\left(  z,z^{\ast}\right)  \in\partial\varphi.
\]
Conversely (only in the case $d=1$) we recall that, if $A$ is a given maximal
monotone operator on $\mathbb{R}$, then there exists a proper l.s.c. convex
function $\psi$ such that $A=\partial\psi.$
\end{remark}

\begin{itemize}
\item[$\mathrm{(H}_{2}\mathrm{)}$] The functions $b:\left[  0,T\right]
\times\mathbb{R}^{3d}\rightarrow\mathbb{R}^{d}$ and $\sigma:\left[
0,T\right]  \times\mathbb{R}^{3d}\rightarrow\mathbb{R}^{d\times n}$ are
continuous and there exist $\ell,\kappa>0$ such that for all $t\in\lbrack0,T]$
and $x,y,z,x^{\prime},y^{\prime},z^{\prime}\in\mathbb{R}^{d}$,%
\begin{equation}%
\begin{array}
[c]{l}%
\left\vert b\left(  t,x,y,z\right)  -b(t,x^{\prime},y^{\prime},z^{\prime
})\right\vert +\left\vert \sigma\left(  t,x,y,z\right)  -\sigma(t,x^{\prime
},y^{\prime},z^{\prime})\right\vert \leq\ell\left(  |x-x^{\prime
}|+|y-y^{\prime}|+|z-z^{\prime}|\right)  ,\medskip\\
\left\vert b\left(  t,0,0,0\right)  \right\vert +\left\vert \sigma\left(
t,0,0,0\right)  \right\vert \leq\kappa.
\end{array}
\label{differ of b, sigma}%
\end{equation}

\item[$\mathrm{(H}_{3}\mathrm{)}$] The initial path $\xi$ is $\mathcal{F}%
_{s}^{s}$-measurable and%
\begin{equation}
\xi\in L^{2}\big(\Omega;\mathcal{C}\big(\left[  -\delta,0\right]
;\overline{\mathrm{Dom}\left(  \varphi\right)  }\,\big)\big)\quad
\text{and}\quad\varphi\left(  \xi\left(  0\right)  \right)  \in L^{1}\left(
\Omega;\mathbb{R}\right)  \,. \label{assumpt on xi}%
\end{equation}

\end{itemize}

\begin{definition}
\label{def sol of delay SVI 1}A pair of progressively measurable continuous
stochastic processes $\left(  X,K\right)  :\Omega\times\lbrack s-\delta
,T]\rightarrow\mathbb{R}^{2d}$ is a solution of (\ref{SVI delayed 2}) if%
\begin{equation}%
\begin{array}
[c]{rl}%
\left(  i\right)  & X\in L_{ad}^{2}\left(  \Omega;C\left(  \left[
s-\delta,T\right]  ;\mathbb{R}^{d}\right)  \right)  ,\medskip\\
\left(  ii\right)  & X\left(  t\right)  \in\overline{\mathrm{Dom}\left(
\varphi\right)  },\;\text{a.e. }t\in\left[  s-\delta,T\right]  ,\;\mathbb{P}%
\text{-a.s. and }\varphi\left(  X\right)  \in L^{1}\left(  \Omega\times\left[
s-\delta,T\right]  ;\mathbb{R}\right)  ,\medskip\\
\left(  iii\right)  & K\in L_{ad}^{2}\left(  \Omega;C\left(  \left[
s,T\right]  ;\mathbb{R}^{d}\right)  \right)  \cap L^{1}\left(  \Omega
;\mathrm{BV}\left(  \left[  s,T\right]  ;\mathbb{R}^{d}\right)  \right)
\text{ with }K\left(  s\right)  =0,\;\mathbb{P}\text{-a.s.},\medskip\\
\left(  iv\right)  & X\left(  t\right)  +K\left(  t\right)  =X\left(
s\right)  +\displaystyle\int_{s}^{t}b\left(  r,X\left(  r\right)  ,Y\left(
r\right)  ,Z\left(  r\right)  \right)  dr\medskip\\
& \quad\quad\quad\quad\quad\quad\quad\displaystyle+\int_{s}^{t}\sigma\left(
r,X\left(  r\right)  ,Y\left(  r\right)  ,Z\left(  r\right)  \right)
dW\left(  r\right)  ,\;\forall t\in(s,T],\;\mathbb{P}\text{-a.s.}\medskip\\
\left(  v\right)  & X\left(  t\right)  =\xi\left(  t-s\right)  ,\;\forall
t\in\lbrack s-\delta,s]\medskip\\
\left(  vi\right)  & \displaystyle\int_{t}^{\hat{t}}\langle u-X\left(
r\right)  ,dK\left(  r\right)  \rangle+\int_{t}^{\hat{t}}\varphi(X\left(
r\right)  )dr\leq(\hat{t}-t)\varphi(u),\;\forall u\,\in\mathbb{R}%
^{d},\;\forall~0\leq t\leq\hat{t}\leq T,\\
& \quad\quad\quad\quad\quad\quad\quad\quad\quad\quad\quad\quad\quad\quad
\quad\quad\quad\quad\quad\quad\quad\quad\quad\quad\quad\quad\quad\quad
\quad\quad\quad\quad\quad\mathbb{P}\text{-a.s.}%
\end{array}
\label{def sol of delay SVI 2}%
\end{equation}

\end{definition}

\begin{remark}
In the following we shall write $dK\left(  t\right)  \in\partial\varphi\left(
Y\left(  t\right)  \right)  dt$, $\mathbb{P}$-a.e. instead of inequality
$\left(  vi\right)  $ (see also the bellow result). Now, taking the processes
$X,X^{\prime},K,K^{\prime}$ such that $dK(t)\in\partial\varphi(X(t))dt$ and
$dK^{\prime}(t)\in\partial\varphi(X^{\prime}\left(  t\right)  )dt$, we see
that%
\begin{equation}
\int_{t}^{\hat{t}}\langle X(r)-X^{\prime}(r),dK(r)-dK^{\prime}(r)\rangle
\geq0,\;\forall~0\leq t\leq\hat{t}\leq T. \label{differ of 2 sol}%
\end{equation}

\end{remark}

Following Proposition 1.2 from \cite{as-ra/97} we can give some equivalent
inequalities with $\left(  vi\right)  :$

\begin{proposition}
If $\varphi:\mathbb{R}^{d}\rightarrow(-\infty,+\infty]$ is a convex and l.s.c.
function and $x\in C\left(  \left[  s,T\right]  ;\mathbb{R}^{d}\right)  $ and
$\eta\in C\left(  \left[  s,T\right]  ;\mathbb{R}^{d}\right)  \cap
\mathrm{BV}\left(  \left[  s,T\right]  ;\mathbb{R}^{d}\right)  .$ The
following assertions are equivalent with the inequality
(\ref{def sol of delay SVI 2}-$vi$):%
\[%
\begin{array}
[c]{rl}%
\left(  vi\right)  ^{\prime} & \displaystyle\int_{t}^{\hat{t}}\langle y\left(
r\right)  -x\left(  r\right)  ,d\eta\left(  r\right)  \rangle+\int_{t}%
^{\hat{t}}\varphi(x\left(  r\right)  )dr\leq\int_{t}^{\hat{t}}\varphi(y\left(
r\right)  )dr,\medskip\\
& \quad\quad\quad\quad\quad\quad\quad\quad\quad\quad\quad\quad\quad\quad
\quad\quad\quad\quad\quad\quad\quad\quad\quad\forall y\in C\left(  \left[
s,T\right]  ;\mathbb{R}^{d}\right)  ,\,\forall~0\leq t\leq\hat{t}\leq
T;\medskip\\
\left(  vi\right)  ^{\prime\prime} & \displaystyle\int_{t}^{\hat{t}}\langle
x\left(  r\right)  -z,d\eta\left(  r\right)  -z^{\ast}\rangle\geq
0,\;\forall\left(  x,x^{\ast}\right)  \in\partial\varphi,\,\forall~0\leq
t\leq\hat{t}\leq T;\medskip\\
\left(  vi\right)  ^{\prime\prime\prime} & \displaystyle\int_{t}^{\hat{t}%
}\langle x\left(  r\right)  -y\left(  r\right)  ,d\eta\left(  r\right)
-y^{\ast}\left(  r\right)  \rangle\geq0,\medskip\\
& \quad\quad\quad\quad\quad\quad\quad\forall y,y^{\ast}\in C\left(  \left[
s,T\right]  ;\mathbb{R}^{d}\right)  ,\;\left(  y\left(  r\right)  ,y^{\ast
}\left(  r\right)  \right)  \in\partial\varphi,\,\forall r\in\left[
-\delta,T\right]  ,\;\forall~0\leq t\leq\hat{t}\leq T;\medskip\\
\left(  vi\right)  ^{\prime\prime\prime\prime} & \displaystyle\int_{0}%
^{T}\langle y\left(  r\right)  -x\left(  r\right)  ,d\eta\left(  r\right)
\rangle+\int_{0}^{T}\varphi(x\left(  r\right)  )dr\leq\int_{0}^{T}%
\varphi(y\left(  r\right)  )dr,\;\forall y\in C\left(  \left[  s,T\right]
;\mathbb{R}^{d}\right)  .
\end{array}
\]

\end{proposition}

In all that follows, $C$ denotes a constant, which may depend only on
$\ell,\kappa,\delta$ and $T$, which may vary from line to line.$\smallskip$

The next result provides some a priori estimates of the solution. Write
$\left\Vert \xi\right\Vert _{\mathcal{C}}=\left\Vert \xi\right\Vert
_{\mathcal{C}([-\delta,0];\mathbb{R}^{d})}$.

\begin{proposition}
\label{a priori estimates}We suppose that assumptions $\mathrm{(H}%
_{1}-\mathrm{H}_{3}\mathrm{)}$ are satisfied. Let $\left(  X,K\right)  $ be a
solution of equation (\ref{SVI delayed 2}). Then there exists a constant
$C=C\left(  \ell,\kappa,\delta,T\right)  >0$ such that%
\[
\mathbb{E}\sup_{r\in\left[  s,T\right]  }\left\vert X\left(  r\right)
\right\vert ^{2}\leq C\,\big(1+\mathbb{E}\left\Vert \xi\right\Vert
_{\mathcal{C}}^{2}\big).
\]
In addition%
\[
\mathbb{E}\sup_{r\in\left[  s,T\right]  }\left\vert Y\left(  r\right)
\right\vert ^{2}+\mathbb{E}\int_{s}^{t}\left\vert Z\left(  r\right)
\right\vert ^{2}dr\leq C\,\big(1+\mathbb{E}\left\Vert \xi\right\Vert
_{\mathcal{C}}^{2}\big).
\]

\end{proposition}

\begin{proof}
Applying It\^{o}'s formula and using (\ref{differ of b, sigma}) and
(\ref{differ of 2 sol}) we obtain%
\begin{equation}%
\begin{array}
[c]{l}%
\left\vert X\left(  t\right)  \right\vert ^{2}\leq\left\vert \xi\left(
0\right)  \right\vert ^{2}+\left(  1+12\ell^{2}\right)  \displaystyle\int
_{s}^{t}\left\vert X\left(  r\right)  \right\vert ^{2}dr+12\ell^{2}\int
_{s}^{t}\left\vert Y\left(  r\right)  \right\vert ^{2}dr+12\ell^{2}\int
_{s}^{t}\left\vert Z\left(  r\right)  \right\vert ^{2}dr\medskip\\
\quad\displaystyle+4\kappa^{2}\left(  t-s\right)  +2\int_{s}^{t}\left\langle
X(r),\sigma\left(  r,X\left(  r\right)  ,Y\left(  r\right)  ,Z\left(
r\right)  \right)  \right\rangle dW\left(  r\right)  ,
\end{array}
\label{Ito for a priori}%
\end{equation}
since, from $\mathrm{(H}_{1}\mathrm{)}$, $0\in\partial\varphi\left(  0\right)
$.

From definition (\ref{def of Y,Z}) we have%
\begin{equation}%
\begin{array}
[c]{l}%
\displaystyle\left\vert Y\left(  r\right)  \right\vert ^{2}=\left\vert
\int_{-\delta}^{0}e^{\lambda u}X\left(  r+u\right)  du\right\vert ^{2}%
\leq\delta\int_{-\delta}^{0}\left\vert X\left(  r+u\right)  \right\vert
^{2}du=\delta\int_{r-\delta}^{r}\left\vert X\left(  u\right)  \right\vert
^{2}du\medskip\\
\displaystyle\leq\delta\left(  \int_{s-\delta}^{s}\left\vert \xi\left(
u-s\right)  \right\vert ^{2}du+\int_{s}^{r}\left\vert X\left(  u\right)
\right\vert ^{2}du\right)  =\delta\left(  \int_{-\delta}^{0}\left\vert
\xi\left(  u\right)  \right\vert ^{2}du+\int_{s}^{r}\left\vert X\left(
u\right)  \right\vert ^{2}du\right)
\end{array}
\label{calc of Y}%
\end{equation}
and%
\begin{equation}%
\begin{array}
[c]{l}%
\displaystyle\int_{s}^{t}\left\vert Z\left(  r\right)  \right\vert ^{2}%
dr=\int_{s}^{t}\left\vert X\left(  r-\delta\right)  \right\vert ^{2}%
dr=\int_{s-\delta}^{t-\delta}\left\vert X\left(  r\right)  \right\vert
^{2}dr\leq\int_{s-\delta}^{s}\left\vert \xi\left(  r-s\right)  \right\vert
^{2}dr+\int_{s}^{t}\left\vert X\left(  r\right)  \right\vert ^{2}dr\medskip\\
\displaystyle=\int_{-\delta}^{0}\left\vert \xi\left(  r\right)  \right\vert
^{2}dr+\int_{s}^{t}\left\vert X\left(  r\right)  \right\vert ^{2}dr
\end{array}
\label{calc of Z}%
\end{equation}
Hence (\ref{Ito for a priori}) becomes%
\begin{equation}%
\begin{array}
[c]{l}%
\left\vert X\left(  t\right)  \right\vert ^{2}\leq\left\vert \xi\left(
0\right)  \right\vert ^{2}+\left(  1+24\ell^{2}+12\ell^{2}\delta T\right)
\displaystyle\int_{s}^{t}\left\vert X\left(  r\right)  \right\vert
^{2}dr+12\ell^{2}\left(  1+\delta T\right)  \int_{-\delta}^{0}\left\vert
\xi\left(  u\right)  \right\vert ^{2}du\medskip\\
\quad\displaystyle+4\kappa^{2}T+2\int_{s}^{t}\left\langle X(r),\sigma\left(
r,X\left(  r\right)  ,Y\left(  r\right)  ,Z\left(  r\right)  \right)
\right\rangle dW\left(  r\right)  .
\end{array}
\label{Ito for a priori 2}%
\end{equation}
From Doob's inequality and (\ref{calc of Y}-\ref{calc of Z}), we deduce that%
\[%
\begin{array}
[c]{l}%
2\mathbb{E}\sup_{r\in\left[  s,t\right]  }\displaystyle\left\vert \int_{s}%
^{r}\left\langle X(u),\sigma\left(  u,X\left(  u\right)  ,Y\left(  u\right)
,Z\left(  u\right)  \right)  \right\rangle dW\left(  u\right)  \right\vert
\medskip\\
\leq6\mathbb{E}\displaystyle\left[  \int_{s}^{t}\left\vert \left\langle
X(u),\sigma\left(  u,X\left(  u\right)  ,Y\left(  u\right)  ,Z\left(
u\right)  \right)  \right\rangle \right\vert ^{2}du\right]  ^{1/2}\medskip\\
\leq6\mathbb{E}\displaystyle\left[  6\ell^{2}\int_{s}^{t}\left\vert
X(u)\right\vert ^{2}\left(  \left\vert X\left(  u\right)  \right\vert
^{2}+\left\vert Y\left(  u\right)  \right\vert ^{2}+\left\vert Z\left(
u\right)  \right\vert ^{2}\right)  du\right]  ^{1/2}+6\mathbb{E}\left[
2\kappa^{2}\int_{s}^{t}\left\vert X(u)\right\vert ^{2}du\right]
^{1/2}\medskip\\
\leq6\sqrt{6}\ell\mathbb{E}\displaystyle\left[  \sup_{r\in\left[  s,t\right]
}\left\vert X\left(  r\right)  \right\vert \left(  \int_{s}^{t}\left\vert
X\left(  r\right)  \right\vert ^{2}dr\right)  ^{1/2}\right]  +6\sqrt{6}%
\ell\mathbb{E}\left[  \int_{s}^{t}\left\vert X\left(  r\right)  \right\vert
^{2}\left\vert Y\left(  r\right)  \right\vert ^{2}dr\right]  ^{1/2}\medskip\\
\quad+6\sqrt{6}\ell\mathbb{E}\displaystyle\left[  \int_{s}^{t}\left\vert
X\left(  r\right)  \right\vert ^{2}\left\vert Z\left(  r\right)  \right\vert
^{2}dr\right]  ^{1/2}+6\sqrt{2}\kappa\mathbb{E}\left[  \int_{s}^{t}\left\vert
X(u)\right\vert ^{2}du\right]  ^{1/2}%
\end{array}
\]
which implies, using Young's inequality$^{\ast}$\footnotetext[1]%
{$\;${\scriptsize $ab\leq\frac{a^{p}}{p}+\frac{b^{q}}{q},\;\forall a,b>0$ and
$\forall p,q>0$ such that $\frac{1}{p}+\frac{1}{q}=1.$}}, that%
\[%
\begin{array}
[c]{l}%
2\mathbb{E}\sup_{r\in\left[  s,t\right]  }\displaystyle\left\vert \int_{s}%
^{r}\left\langle X(u),\sigma\left(  u,X\left(  u\right)  ,Y\left(  u\right)
,Z\left(  u\right)  \right)  \right\rangle dW\left(  u\right)  \right\vert
\medskip\\
\leq\displaystyle\frac{1}{4}\mathbb{E}\sup_{r\in\left[  s,t\right]
}\left\vert X\left(  r\right)  \right\vert ^{2}+\left(  216\ell^{2}+3\sqrt
{2}\kappa\right)  \mathbb{E}\int_{s}^{t}\left\vert X\left(  r\right)
\right\vert ^{2}dr+6\sqrt{6\delta}\ell\mathbb{E}\left[  \int_{s}^{t}\left\vert
X\left(  r\right)  \right\vert ^{2}dr\cdot\int_{-\delta}^{0}\left\vert
\xi\left(  u\right)  \right\vert ^{2}du\right]  ^{1/2}\medskip\\
\quad\displaystyle+6\sqrt{6\delta}\ell\mathbb{E}\left[  \int_{s}^{t}\left\vert
X\left(  r\right)  \right\vert ^{2}\left(  \int_{s}^{r}\left\vert X\left(
u\right)  \right\vert ^{2}du\right)  dr\right]  ^{1/2}+6\sqrt{6}\ell
\mathbb{E}\left[  \int_{s}^{t}\left\vert X\left(  r\right)  \right\vert
^{2}\left\vert X\left(  r-\delta\right)  \right\vert ^{2}dr\right]
^{1/2}+3\sqrt{2}\kappa\medskip\\
\leq\displaystyle\frac{1}{4}\mathbb{E}\sup_{r\in\left[  s,t\right]
}\left\vert X\left(  r\right)  \right\vert ^{2}+C_{1}\mathbb{E}\int_{s}%
^{t}\left\vert X\left(  r\right)  \right\vert ^{2}dr+C_{2}\mathbb{E}%
\int_{-\delta}^{0}\left\vert \xi\left(  u\right)  \right\vert ^{2}du\medskip\\
\quad\displaystyle+6\sqrt{6}\ell\mathbb{E}\left[  \sup_{r\in\left[
s,t\right]  }\left\vert X\left(  r\right)  \right\vert \left(  \int_{s-\delta
}^{t-\delta}\left\vert X\left(  r\right)  \right\vert ^{2}dr\right)
^{1/2}\right]  +3\sqrt{2}\kappa\medskip\\
\leq\displaystyle\frac{1}{2}\mathbb{E}\sup_{r\in\left[  s,t\right]
}\left\vert X\left(  r\right)  \right\vert ^{2}+C_{3}\mathbb{E}\int_{s}%
^{t}\left\vert X\left(  r\right)  \right\vert ^{2}dr+C_{4}\mathbb{E}%
\int_{-\delta}^{0}\left\vert \xi\left(  u\right)  \right\vert ^{2}du+3\sqrt
{2}\kappa.
\end{array}
\]
Therefore, from (\ref{Ito for a priori 2}), we deduce that there exists
another constant $C>0$ such that%
\[
\mathbb{E}\sup_{r\in\left[  s,t\right]  }\left\vert X\left(  r\right)
\right\vert ^{2}\leq C\,\big(1+\mathbb{E}\left\Vert \xi\right\Vert
_{\mathcal{C}}^{2}\big)+C\int_{s}^{t}\mathbb{E}\sup_{u\in\left[  s,r\right]
}\left\vert X\left(  u\right)  \right\vert ^{2}dr
\]
and from Gronwall's inequality we obtain the conclusion.\hfill\medskip
\end{proof}

The next result emphasize the continuous dependence of the solution $(X,K)$
with respect to the initial values $\left(  s,\xi\right)  $. Obviously, the
uniqueness of the solution will be a immediate consequence.

\begin{proposition}
We suppose that assumptions $\mathrm{(H}_{1}-\mathrm{H}_{3}\mathrm{)}$ are
satisfied. If $\left(  X^{s,\xi},K^{s,\xi}\right)  $ and $(X^{s^{\prime}%
,\xi^{\prime}},K^{s^{\prime},\xi^{\prime}})$ are the solutions of
(\ref{SVI delayed 2}) corresponding to the initial data $\left(  s,\xi\right)
$ and $(s^{\prime},\xi^{\prime})$ respectively, then there exists $C=C\left(
\ell,\kappa,\delta,T\right)  >0$ such that%
\begin{equation}%
\begin{array}
[c]{r}%
\mathbb{E}\sup_{r\in\left[  s\wedge s^{\prime},t\right]  }|X^{s,\xi}\left(
r\right)  -X^{s^{\prime},\xi^{\prime}}\left(  r\right)  |^{2}+\mathbb{E}%
\sup_{r\in\left[  s\wedge s^{\prime},t\right]  }|K^{s,\xi}\left(  r\right)
-K^{s^{\prime},\xi^{\prime}}\left(  r\right)  |^{2}\medskip\\
\leq C\,\Big[\Gamma_{1}+\left\vert s-s^{\prime}\right\vert \big(1+\mathbb{E}%
\left\Vert \xi\right\Vert _{\mathcal{C}}^{2}+\mathbb{E}||\xi^{\prime
}||_{\mathcal{C}}^{2}\big)\Big],
\end{array}
\label{ineq uniq}%
\end{equation}
where%
\begin{equation}
\Gamma_{1}:=\mathbb{E}||\xi-\xi^{\prime}||_{\mathcal{C}}^{2}+\mathbb{E}%
\int_{s^{\prime}-\delta}^{s^{\prime}}\left\vert \xi^{\prime}\left(
r-s\right)  -\xi^{\prime}\left(  r-s^{\prime}\right)  \right\vert ^{2}dr.
\label{def gamma 1}%
\end{equation}

\end{proposition}

\begin{proof}
For convenience we suppose that $s^{\prime}\leq s$.

Since $\left(  X^{s,\xi},K^{s,\xi}\right)  $ and $(X^{s^{\prime},\xi^{\prime}%
},K^{s^{\prime},\xi^{\prime}})$ are the solutions, $\forall t\in
(s,T],\;\mathbb{P}$-a.s.,%
\begin{equation}%
\begin{array}
[c]{r}%
X^{s^{\prime},\xi^{\prime}}\left(  t\right)  +K^{s^{\prime},\xi^{\prime}%
}\left(  t\right)  =X^{s^{\prime},\xi^{\prime}}\left(  s\right)
+\displaystyle\int_{s}^{t}b\big(r,X^{s^{\prime},\xi^{\prime}}\left(  r\right)
,Y^{s^{\prime},\xi^{\prime}}\left(  r\right)  ,Z^{s^{\prime},\xi^{\prime}%
}\left(  r\right)  \big)dr\medskip\\
\displaystyle+\int_{s}^{t}\sigma\big(r,X^{s^{\prime},\xi^{\prime}}\left(
r\right)  ,Y^{s^{\prime},\xi^{\prime}}\left(  r\right)  ,Z^{s^{\prime}%
,\xi^{\prime}}\left(  r\right)  \big)dW\left(  r\right)
\end{array}
\label{SVI delayed 3}%
\end{equation}
and%
\begin{equation}%
\begin{array}
[c]{r}%
X^{s,\xi}\left(  t\right)  +K^{s,\xi}\left(  t\right)  =X^{s,\xi}\left(
s\right)  +\displaystyle\int_{s}^{t}b\big(r,X^{s,\xi}\left(  r\right)
,Y^{s,\xi}\left(  r\right)  ,Z^{s,\xi}\left(  r\right)  \big)dr\medskip\\
\displaystyle+\int_{s}^{t}\sigma\big(r,X^{s,\xi}\left(  r\right)  ,Y^{s,\xi
}\left(  r\right)  ,Z^{s,\xi}\left(  r\right)  \big)dW\left(  r\right)  .
\end{array}
\label{SVI delayed 4}%
\end{equation}
Applying It\^{o}'s formula to $\big|X^{s,\xi}\left(  t\right)  -X^{s^{\prime
},\xi^{\prime}}\left(  t\right)  \big|^{2}$, we have%
\[%
\begin{array}
[c]{l}%
\left\vert \Delta X(t)\right\vert ^{2}+2\displaystyle\int_{s}^{t}%
\big\langle\Delta X(r),dK^{s,\xi}\left(  r\right)  -dK^{s^{\prime},\xi
^{\prime}}\left(  r\right)  \big\rangle=|\Delta X(s)|^{2}+2\int_{s}%
^{t}\left\langle \Delta X(r),b(r)-b^{\prime}(r)\right\rangle dr\medskip\\
\quad+\displaystyle\int_{s}^{t}\left\vert \sigma(r)-\sigma^{\prime
}(r)\right\vert ^{2}dr+2\int_{s}^{t}\left\langle \Delta X(r),\sigma
(r)-\sigma^{\prime}(r)\right\rangle dW\left(  r\right)  ,
\end{array}
\]
where%
\[%
\begin{array}
[c]{l}%
\Delta X(r)=X^{s,\xi}\left(  r\right)  -X^{s^{\prime},\xi^{\prime}}\left(
r\right)  ,\medskip\\
b\left(  r\right)  =b\left(  r,X^{s,\xi}\left(  r\right)  ,Y^{s,\xi}\left(
r\right)  ,Z^{s,\xi}\left(  r\right)  \right)  ,\quad b^{\prime}\left(
r\right)  =b\big(r,X^{s^{\prime},\xi^{\prime}}\left(  r\right)  ,Y^{s^{\prime
},\xi^{\prime}}\left(  r\right)  ,Z^{s^{\prime},\xi^{\prime}}\left(  r\right)
\big),\medskip\\
\sigma\left(  r\right)  =\sigma\left(  r,X^{s,\xi}\left(  r\right)  ,Y^{s,\xi
}\left(  r\right)  ,Z^{s,\xi}\left(  r\right)  \right)  ,\quad\sigma^{\prime
}\left(  r\right)  =\sigma\big(r,X^{s^{\prime},\xi^{\prime}}\left(  r\right)
,Y^{s^{\prime},\xi^{\prime}}\left(  r\right)  ,Z^{s^{\prime},\xi^{\prime}%
}\left(  r\right)  \big).
\end{array}
\]
Using (\ref{differ of b, sigma}) and (\ref{differ of 2 sol}) we see that%
\begin{equation}%
\begin{array}
[c]{l}%
\left\vert \Delta X\left(  t\right)  \right\vert ^{2}\leq|\Delta
X(s)|^{2}+\left(  1+6\ell^{2}\right)  \displaystyle\int_{s}^{t}\left\vert
\Delta X\left(  r\right)  \right\vert ^{2}dr+3\ell^{2}\displaystyle\int
_{s}^{t}\left\vert \Delta Y\left(  r\right)  \right\vert ^{2}dr\medskip\\
\quad+3\ell^{2}\displaystyle\int_{s}^{t}\left\vert \Delta Z\left(  r\right)
\right\vert ^{2}dr+2\int_{s}^{t}\left\langle \Delta X(r),\sigma(r)-\sigma
^{\prime}(r)\right\rangle dW\left(  r\right)  ,
\end{array}
\label{Ito for differ}%
\end{equation}
where%
\[
\Delta Y(r)=Y^{s,\xi}\left(  r\right)  -Y^{s^{\prime},\xi^{\prime}}\left(
r\right)  ,\quad\Delta Z(r)=Z^{s,\xi}\left(  r\right)  -Z^{s^{\prime}%
,\xi^{\prime}}\left(  r\right)  .
\]
Using definition (\ref{def of Y,Z}), we deduce, as in (\ref{calc of Y}) and
(\ref{calc of Z}), that%
\begin{equation}%
\begin{array}
[c]{l}%
\displaystyle\left\vert \Delta Y\left(  r\right)  \right\vert ^{2}\leq
\delta\int_{s-\delta}^{s}\left\vert \Delta X\left(  u\right)  \right\vert
^{2}du+\delta\int_{s}^{r}\left\vert \Delta X\left(  u\right)  \right\vert
^{2}du\medskip\\
\displaystyle\leq\delta\int_{s-\delta}^{\left(  s-\delta\right)  \vee
s^{\prime}}\left\vert \xi\left(  r-s\right)  -\xi^{\prime}\left(  r-s^{\prime
}\right)  \right\vert ^{2}dr+\delta\int_{s^{\prime}}^{s}\left\vert \Delta
X\left(  u\right)  \right\vert ^{2}du+\delta\int_{s}^{r}\left\vert \Delta
X\left(  u\right)  \right\vert ^{2}du
\end{array}
\label{differ of Y}%
\end{equation}
and%
\begin{equation}%
\begin{array}
[c]{l}%
\displaystyle\int_{s}^{t}\left\vert \Delta Z\left(  r\right)  \right\vert
^{2}dr\leq\int_{s-\delta}^{s}\left\vert \Delta X\left(  r\right)  \right\vert
^{2}dr+\int_{s}^{t}\left\vert \Delta X\left(  r\right)  \right\vert
^{2}dr\medskip\\
\displaystyle\leq\int_{s-\delta}^{\left(  s-\delta\right)  \vee s^{\prime}%
}\left\vert \xi\left(  r-s\right)  -\xi^{\prime}\left(  r-s^{\prime}\right)
\right\vert ^{2}dr+\int_{s^{\prime}}^{s}\left\vert \Delta X\left(  r\right)
\right\vert ^{2}dr+\int_{s}^{t}\left\vert \Delta X\left(  r\right)
\right\vert ^{2}dr.
\end{array}
\label{differ of Z}%
\end{equation}
Now, since%
\[%
\begin{array}
[c]{r}%
\Delta X\left(  t\right)  +K^{s,\xi}\left(  t\right)  -K^{s^{\prime}%
,\xi^{\prime}}\left(  t\right)  =\Big[\xi\left(  0\right)  -\xi^{\prime
}\left(  0\right)  -\displaystyle\int_{s^{\prime}}^{s}b^{\prime}%
(r)dr-\int_{s^{\prime}}^{s}\sigma^{\prime}(r)dW\left(  r\right)
\Big]\medskip\\
+\displaystyle\int_{s}^{t}\left[  b(r)-b^{\prime}(r)\right]  dr+\int_{s}%
^{t}\left[  \sigma(r)-\sigma^{\prime}(r)\right]  dW\left(  r\right)
,\;\forall t\in\left[  s,T\right]  ,
\end{array}
\]
the It\^{o}'s formula and (\ref{differ of 2 sol}) yields%
\[%
\begin{array}
[c]{l}%
\displaystyle|\Delta X(s)|^{2}\leq3|\xi\left(  0\right)  -\xi^{\prime}\left(
0\right)  |^{2}+3\left\vert \int_{s^{\prime}}^{s}b^{\prime}\left(  r\right)
dr\right\vert ^{2}+3\left\vert \int_{s^{\prime}}^{s}\sigma^{\prime}\left(
r\right)  dW\left(  r\right)  \right\vert ^{2}\medskip\\
\displaystyle\leq3|\xi\left(  0\right)  -\xi^{\prime}\left(  0\right)
|^{2}+6\left(  s-s^{\prime}\right)  \int_{s^{\prime}}^{s}\left\vert b^{\prime
}\left(  r\right)  -b\left(  r,0,0,0\right)  \right\vert ^{2}dr+6\left(
s-s^{\prime}\right)  ^{2}\kappa^{2}+3\left\vert \int_{s^{\prime}}^{s}%
\sigma^{\prime}\left(  r\right)  dW\left(  r\right)  \right\vert ^{2}%
\medskip\\
\displaystyle\leq3|\xi\left(  0\right)  -\xi^{\prime}\left(  0\right)
|^{2}+6\left(  s-s^{\prime}\right)  ^{2}\kappa^{2}+18\ell^{2}\left(
s-s^{\prime}\right)  \int_{s^{\prime}}^{s}\left(  |X^{s^{\prime},\xi^{\prime}%
}\left(  r\right)  |^{2}+|Y^{s^{\prime},\xi^{\prime}}\left(  r\right)
|^{2}+|Z^{s^{\prime},\xi^{\prime}}\left(  r\right)  |^{2}\right)  dr\medskip\\
\displaystyle\quad+3\left\vert \int_{\bar{s}}^{s}\sigma^{\prime}\left(
r\right)  dW\left(  r\right)  \right\vert ^{2}.
\end{array}
\]
Hence (\ref{Ito for differ}) becomes%
\begin{equation}%
\begin{array}
[c]{l}%
\displaystyle\left\vert \Delta X\left(  t\right)  \right\vert ^{2}\leq
\max\left(  3,6\ell^{2}\left(  1+\delta T\right)  \right)  \Gamma_{1}+6\left(
s-s^{\prime}\right)  ^{2}\kappa^{2}+3\left\vert \int_{s^{\prime}}^{s}%
\sigma^{\prime}\left(  r\right)  dW\left(  r\right)  \right\vert ^{2}%
\medskip\\
\displaystyle\quad+18\ell^{2}\left(  s-s^{\prime}\right)  \int_{s^{\prime}%
}^{s}\left(  |X^{s^{\prime},\xi^{\prime}}\left(  r\right)  |^{2}%
+|Y^{s^{\prime},\xi^{\prime}}\left(  r\right)  |^{2}+|Z^{s^{\prime}%
,\xi^{\prime}}\left(  r\right)  |^{2}\right)  dr\medskip\\
\displaystyle\quad+3\ell^{2}\left(  1+\delta T\right)  \int_{s^{\prime}}%
^{s}\left\vert \Delta X\left(  r\right)  \right\vert ^{2}dr+\left(
1+9\ell^{2}+3\ell^{2}\delta T\right)  \int_{s}^{t}\left\vert \Delta X\left(
r\right)  \right\vert ^{2}dr\medskip\\
\displaystyle\quad+2\int_{s}^{t}\left\langle \Delta X(r),\sigma(r)-\sigma
^{\prime}(r)\right\rangle dW\left(  r\right)  ,
\end{array}
\label{Ito for differ 2}%
\end{equation}
since%
\[%
\begin{array}
[c]{l}%
\displaystyle\int_{s-\delta}^{\left(  s-\delta\right)  \vee s^{\prime}%
}\left\vert \xi\left(  r-s\right)  -\xi^{\prime}\left(  r-s^{\prime}\right)
\right\vert ^{2}dr\medskip\\
\displaystyle\leq2\int_{s-\delta}^{s}\left\vert \xi\left(  r-s\right)
-\xi^{\prime}\left(  r-s\right)  \right\vert ^{2}dr+2\int_{s^{\prime}-\delta
}^{s^{\prime}}\left\vert \xi^{\prime}\left(  r-s\right)  -\xi^{\prime}\left(
r-s^{\prime}\right)  \right\vert ^{2}dr\medskip\\
\displaystyle\leq2\int_{-\delta}^{0}\left\vert \xi\left(  r\right)
-\xi^{\prime}\left(  r\right)  \right\vert ^{2}dr+2\int_{s^{\prime}-\delta
}^{s^{\prime}}\left\vert \xi^{\prime}\left(  r-s\right)  -\xi^{\prime}\left(
r-s^{\prime}\right)  \right\vert ^{2}dr.
\end{array}
\]
By Doob inequality, we deduce that%
\[%
\begin{array}
[c]{l}%
2\mathbb{E}\sup_{r\in\left[  s,t\right]  }\displaystyle\left\vert \int_{s}%
^{r}\left\langle \Delta X(u),\sigma(u)-\sigma^{\prime}(u)\right\rangle
dW\left(  u\right)  \right\vert \leq6\mathbb{E}\left[  \int_{s}^{t}\left\vert
\left\langle \Delta X(u),\sigma(u)-\sigma^{\prime}(u)\right\rangle \right\vert
^{2}du\right]  ^{1/2}\medskip\\
\leq6\sqrt{3}\ell\mathbb{E}\displaystyle\left[  \int_{s}^{t}\left\vert \Delta
X\left(  u\right)  \right\vert ^{2}\left(  \left\vert \Delta X\left(
u\right)  \right\vert ^{2}+\left\vert \Delta Y\left(  u\right)  \right\vert
^{2}+\left\vert \Delta Z\left(  u\right)  \right\vert ^{2}\right)  du\right]
^{1/2}\medskip\\
\leq6\sqrt{3}\ell\mathbb{E}\displaystyle\left[  \sup_{u\in\left[  s,t\right]
}\left\vert \Delta X\left(  u\right)  \right\vert \left(  \int_{s}%
^{t}\left\vert \Delta X\left(  u\right)  \right\vert ^{2}du\right)
^{1/2}\right]  +6\sqrt{3}\ell\mathbb{E}\left[  \int_{s}^{t}\left\vert \Delta
X\left(  u\right)  \right\vert ^{2}\left\vert \Delta Y\left(  u\right)
\right\vert ^{2}du\right]  ^{1/2}\medskip\\
\quad+6\sqrt{3}\ell\mathbb{E}\displaystyle\left[  \int_{s}^{t}\left\vert
\Delta X\left(  u\right)  \right\vert ^{2}\left\vert \Delta Z\left(  u\right)
\right\vert ^{2}du\right]  ^{1/2}%
\end{array}
\]
and, using also the computations from (\ref{differ of Y}) and
(\ref{differ of Z}),%
\[%
\begin{array}
[c]{l}%
2\mathbb{E}\sup_{r\in\left[  s,t\right]  }\displaystyle\left\vert \int_{s}%
^{r}\left\langle \Delta X(u),\sigma(u)-\sigma^{\prime}(u)\right\rangle
dW\left(  u\right)  \right\vert \medskip\\
\leq\displaystyle\frac{1}{4}\mathbb{E}\sup_{r\in\left[  s,t\right]
}\left\vert \Delta X\left(  r\right)  \right\vert ^{2}+108\ell^{2}%
\mathbb{E}\int_{s}^{t}\left\vert \Delta X\left(  r\right)  \right\vert
^{2}dr+6\sqrt{6\delta}\ell\mathbb{E}\left[  \int_{s}^{t}\left\vert \Delta
X\left(  r\right)  \right\vert ^{2}\left(  \int_{-\delta}^{0}\left\vert
\xi\left(  u\right)  -\xi^{\prime}\left(  u\right)  \right\vert ^{2}du\right)
dr\right]  ^{1/2}\medskip\\
\quad\displaystyle+6\sqrt{6\delta}\ell\mathbb{E}\left[  \int_{s}^{t}\left\vert
\Delta X\left(  r\right)  \right\vert ^{2}\left(  \int_{s^{\prime}-\delta
}^{s^{\prime}}\left\vert \xi^{\prime}\left(  u-s\right)  -\xi^{\prime}\left(
u-s^{\prime}\right)  \right\vert ^{2}du\right)  dr\right]  ^{1/2}\medskip\\
\quad\displaystyle+6\sqrt{3\delta}\ell\mathbb{E}\left[  \int_{s}^{t}\left\vert
\Delta X\left(  r\right)  \right\vert ^{2}\left(  \int_{\bar{s}}^{s}\left\vert
\Delta X\left(  u\right)  \right\vert ^{2}du\right)  du\right]  ^{1/2}%
\medskip\\
\quad\displaystyle+6\sqrt{3\delta}\ell\mathbb{E}\left[  \int_{s}^{t}\left\vert
\Delta X\left(  r\right)  \right\vert ^{2}\left(  \int_{s}^{r}\left\vert
\Delta X\left(  r\right)  \right\vert ^{2}dr\right)  dr\right]  ^{1/2}%
+6\sqrt{3}\ell\mathbb{E}\left[  \int_{s}^{t}\left\vert \Delta X\left(
r\right)  \right\vert ^{2}\left\vert \Delta Z\left(  r\right)  \right\vert
^{2}dr\right]  ^{1/2}\medskip\\
\leq\displaystyle\frac{1}{2}\mathbb{E}\sup_{r\in\left[  s,t\right]
}\left\vert \Delta X\left(  r\right)  \right\vert ^{2}+C_{1}\mathbb{E}%
\int_{-\delta}^{0}\left\vert \xi\left(  r\right)  -\xi^{\prime}\left(
r\right)  \right\vert ^{2}dr+C_{2}\mathbb{E}\int_{s^{\prime}-\delta
}^{s^{\prime}}\left\vert \xi^{\prime}\left(  r-s\right)  -\xi^{\prime}\left(
r-s^{\prime}\right)  \right\vert ^{2}dr\medskip\\
\quad\displaystyle+C_{3}\mathbb{E}\int_{\bar{s}}^{s}\left\vert \Delta X\left(
r\right)  \right\vert ^{2}dr+C_{4}\mathbb{E}\int_{s}^{t}\left\vert \Delta
X\left(  r\right)  \right\vert ^{2}dr.
\end{array}
\]
From inequality (\ref{Ito for differ 2}) and the inequalities obtained in
Proposition \ref{a priori estimates} we see that there exists a constant $C>0$
such that%
\[%
\begin{array}
[c]{l}%
\displaystyle\mathbb{E}\sup_{r\in\left[  s,t\right]  }\left\vert \Delta
X\left(  r\right)  \right\vert ^{2}\leq C\,\Gamma_{1}+C\left(  s-s^{\prime
}\right)  +3\mathbb{E}\int_{s^{\prime}}^{s}\left\vert \sigma^{\prime}\left(
r\right)  \right\vert ^{2}dr\medskip\\
\displaystyle\quad+18\ell^{2}\left(  s-s^{\prime}\right)  \mathbb{E}%
\int_{s^{\prime}}^{s}\left(  |X^{s^{\prime},\xi^{\prime}}\left(  r\right)
|^{2}+|Y^{s^{\prime},\xi^{\prime}}\left(  r\right)  |^{2}+|Z^{s^{\prime}%
,\xi^{\prime}}\left(  r\right)  |^{2}\right)  dr\medskip\\
\displaystyle\quad+3\ell^{2}\left(  1+\delta T\right)  \mathbb{E}%
\int_{s^{\prime}}^{s}\left\vert \Delta X\left(  r\right)  \right\vert
^{2}dr+C\int_{s}^{t}\mathbb{E}\sup_{u\in\left[  s,r\right]  }\left\vert \Delta
X\left(  u\right)  \right\vert ^{2}dr\medskip\\
\displaystyle\leq C~\Gamma_{1}+C\left(  s-s^{\prime}\right)  +C\int_{s}%
^{t}\mathbb{E}\sup_{u\in\left[  s,r\right]  }\left\vert \Delta X\left(
u\right)  \right\vert ^{2}dr\medskip\\
\displaystyle\quad+C\mathbb{E}\int_{s^{\prime}}^{s}\left(  |X^{s,\xi}\left(
r\right)  |^{2}+|X^{s^{\prime},\xi^{\prime}}\left(  r\right)  |^{2}+|Y^{s,\xi
}\left(  r\right)  |^{2}+|Y^{s^{\prime},\xi^{\prime}}\left(  r\right)
|^{2}+|Z^{s,\xi}\left(  r\right)  |^{2}+|Z^{s^{\prime},\xi^{\prime}}\left(
r\right)  |^{2}\right)  dr\medskip\\
\displaystyle\leq C~\Big[\Gamma_{1}+\left\vert s-s^{\prime}\right\vert
\big(1+\mathbb{E}\left\Vert \xi\right\Vert _{\mathcal{C}}^{2}+\mathbb{E}%
||\xi^{\prime}||_{\mathcal{C}}^{2}\big)\Big]+C\int_{s}^{t}\mathbb{E}\sup
_{u\in\left[  s,r\right]  }\left\vert \Delta X\left(  u\right)  \right\vert
^{2}dr
\end{array}
\]
and therefore, applying Gronwall's inequality,%
\[
\mathbb{E}\sup_{r\in\left[  s,t\right]  }|X^{s,\xi}\left(  r\right)
-X^{s^{\prime},\xi^{\prime}}\left(  r\right)  |^{2}\leq C~\Big[\Gamma
_{1}+\left\vert s-s^{\prime}\right\vert \big(1+\mathbb{E}\left\Vert
\xi\right\Vert _{\mathcal{C}}^{2}+\mathbb{E}||\xi^{\prime}||_{\mathcal{C}}%
^{2}\big)\Big].
\]
In order to finish the proof of (\ref{ineq uniq}) we shall use the above
inequalities and similar computations in the following inequality which is
obtained from (\ref{SVI delayed 3}) and (\ref{SVI delayed 4}):%
\[%
\begin{array}
[c]{l}%
\mathbb{E}\sup_{r\in\left[  s,T\right]  }|K^{s,\xi}\left(  r\right)
-K^{s^{\prime},\xi^{\prime}}\left(  r\right)  |^{2}\leq4|X^{s,\xi}\left(
s\right)  -X^{s^{\prime},\xi^{\prime}}\left(  s\right)  |^{2}+4\mathbb{E}%
\sup_{r\in\left[  s,T\right]  }\left\vert \Delta X\left(  r\right)
\right\vert ^{2}\medskip\\
\quad+4T\displaystyle\mathbb{E}\int_{s}^{t}\left\vert b\left(  r\right)
-b^{\prime}\left(  r\right)  \right\vert ^{2}dr+4\mathbb{E}\sup_{r\in\left[
s,T\right]  }\left\vert \int_{s}^{r}\left(  \sigma\left(  r\right)
-\sigma^{\prime}\left(  r\right)  \right)  dW\left(  r\right)  \right\vert
^{2}.
\end{array}
\]
\hfill
\end{proof}

We state now the main result of this section:

\begin{theorem}
\label{main result 1}Under the assumptions $\mathrm{(H}_{1}-\mathrm{H}%
_{3}\mathrm{)}$ equation (\ref{SVI delayed 2}) has a unique solution in the
sense of Definition \ref{def sol of delay SVI 1}. Moreover, there exists a
constant $C=C\left(  \ell,\kappa,\delta,T\right)  >0$ such that%
\[
\mathbb{E}\sup_{r\in\left[  s,T\right]  }\left\vert X\left(  r\right)
\right\vert ^{2}+\mathbb{E}\sup_{r\in\left[  s,T\right]  }\left\vert K\left(
r\right)  \right\vert ^{2}+\mathbb{E}\left\Vert K\right\Vert _{\mathrm{BV}%
\left(  \left[  -\delta,T\right]  ;\mathbb{R}^{d}\right)  }+\mathbb{E}\int
_{s}^{T}\varphi\left(  X\left(  r\right)  \right)  dr\leq C\,\big(1+\mathbb{E}%
\left\Vert \xi\right\Vert _{\mathcal{C}}^{2}\big)
\]
and%
\[
\mathbb{E}\sup_{r\in\left[  s,T\right]  }\left\vert X\left(  r\right)
\right\vert ^{4}+\mathbb{E}\left\Vert K\right\Vert _{\mathrm{BV}\left(
\left[  -\delta,T\right]  ;\mathbb{R}^{d}\right)  }^{2}+\mathbb{E}%
\Big(\int_{s}^{T}\varphi\left(  X\left(  r\right)  \right)  dr\Big)^{2}\leq
C\,\big(1+\mathbb{E}\left\Vert \xi\right\Vert _{\mathcal{C}}^{4}\big).
\]

\end{theorem}

\subsection{Proof of Theorem \ref{main result 1}}

In order to simplify computations we will assume that $s=0$.

The existence of a solution will be proved starting from the penalized
equation:%
\begin{equation}
\left\{
\begin{array}
[c]{r}%
dX_{\epsilon}\left(  t\right)  +\nabla\varphi_{\epsilon}\left(  X_{\epsilon
}\left(  t\right)  \right)  dt=b\left(  t,X_{\epsilon}\left(  t\right)
,Y_{\epsilon}\left(  t\right)  ,Z_{\epsilon}\left(  t\right)  \right)
dt+\sigma\left(  t,X_{\epsilon}\left(  t\right)  ,Y_{\epsilon}\left(
t\right)  ,Z_{\epsilon}\left(  t\right)  \right)  dW(t),\medskip\\
t\in(0,T],\medskip\\
\multicolumn{1}{l}{X_{\epsilon}\left(  t\right)  =\xi\left(  t\right)
,\;t\in\left[  -\delta,0\right]  ,}%
\end{array}
\right.  \label{SVI delayed approxim}%
\end{equation}
where, for $\epsilon>0$, $\varphi_{\epsilon}$ is defined by
(\ref{Yosida regular}) and%
\begin{equation}
Y_{\epsilon}(t):=\int_{-\delta}^{0}e^{\lambda r}X_{\epsilon}(t+r)dr,\quad
Z_{\epsilon}(t):=X_{\epsilon}(t-\delta). \label{def of Y,Z eps}%
\end{equation}
The proof will be spitted into several steps which are adapted from the proof
of Theorem 2.1 from \cite{as-ra/97}.

Since $\nabla\varphi_{\epsilon}$ is a Lipschitz function, it is known (see
e.g. \cite{mo/98}) that there exists a unique solution $X_{\epsilon}\in
L_{ad}^{2}\left(  \Omega;C\left(  \left[  -\delta,T\right]  \right)  \right)
$.

We define%
\begin{equation}
K_{\epsilon}\left(  t\right)  =\int_{0}^{t}\nabla\varphi_{\epsilon}\left(
X_{\epsilon}\left(  s\right)  \right)  ds. \label{def of K eps}%
\end{equation}
Taking into account that $\mathrm{(H}_{1}\mathrm{)}$ is satisfied, we can
assume in what follows, without restrict our generality, that%
\[
\xi\in L^{\infty}\big(\Omega;\mathcal{C}\big(\left[  -\delta,0\right]
;\mathrm{Int}\left(  \mathrm{Dom}\left(  \varphi\right)  \right)
\big)\big)\quad\text{and}\quad\varphi\left(  \xi\left(  0\right)  \right)  \in
L^{\infty}\left(  \Omega;\mathbb{R}\right)  .
\]
$\smallskip$

\noindent\textrm{A.} \textit{Boundedness of }$X_{\epsilon}$\textit{ and
}$K_{\epsilon}$

We will prove that there exists a constant $C>0$ such that%
\begin{equation}
\mathbb{E}\sup_{s\in\left[  0,T\right]  }\left\vert X_{\epsilon}\left(
s\right)  \right\vert ^{4}+\displaystyle\mathbb{E}\left[  \int_{0}^{T}%
\varphi\left(  J_{\epsilon}\left(  X_{\epsilon}\left(  r\right)  \right)
\right)  dr\right]  ^{2}+\mathbb{E}\left[  \int_{0}^{T}\varphi_{\epsilon
}\left(  X_{\epsilon}\left(  r\right)  \right)  dr\right]  ^{2}\leq
C~\big[1+\mathbb{E}\left\Vert \xi\right\Vert _{\mathcal{C}}^{4}%
\big] \label{ineq bound}%
\end{equation}
and%
\begin{equation}
\mathbb{E}\sup_{s\in\left[  0,T\right]  }\left\vert X_{\epsilon}\left(
s\right)  \right\vert ^{2}+\displaystyle\mathbb{E}\left[  \int_{0}^{T}%
\varphi_{\epsilon}\left(  X_{\epsilon}\left(  r\right)  \right)  dr\right]
\leq C~\big[1+\mathbb{E}\left\Vert \xi\right\Vert _{\mathcal{C}}^{4}\big].
\label{ineq bound 2}%
\end{equation}
Indeed, by applying It\^{o}'s formula we see that%
\begin{equation}%
\begin{array}
[c]{l}%
\left\vert X_{\epsilon}\left(  t\right)  \right\vert ^{2}+2\displaystyle\int
_{0}^{t}\left\langle X_{\epsilon}\left(  r\right)  ,\nabla\varphi_{\epsilon
}\left(  X_{\epsilon}\left(  r\right)  \right)  \right\rangle dr=\left\vert
X\left(  0\right)  \right\vert ^{2}+\int_{0}^{t}\left\vert \sigma_{\epsilon
}\left(  r\right)  \right\vert ^{2}dr\medskip\\
\quad+2\displaystyle\int_{0}^{t}\left\langle X_{\epsilon}(r),b_{\epsilon
}\left(  r\right)  \right\rangle dr+2\int_{0}^{t}\left\langle X_{\epsilon
}(r),\sigma_{\epsilon}\left(  r\right)  \right\rangle dW(r),
\end{array}
\label{Ito for approx}%
\end{equation}
where%
\begin{equation}
b_{\epsilon}\left(  r\right)  =b\left(  r,X_{\epsilon}\left(  r\right)
,Y_{\epsilon}\left(  r\right)  ,Z_{\epsilon}\left(  r\right)  \right)
,\quad\sigma_{\epsilon}\left(  r\right)  =\sigma\left(  r,X_{\epsilon}\left(
r\right)  ,Y_{\epsilon}\left(  r\right)  ,Z_{\epsilon}\left(  r\right)
\right)  . \label{def of b,sigma eps}%
\end{equation}
Using (\ref{ineq Yosida 2}-$ix$) we get%
\begin{equation}%
\begin{array}
[c]{l}%
\left\vert X_{\epsilon}\left(  t\right)  \right\vert ^{4}%
+4\displaystyle\left(  \int_{0}^{t}\varphi_{\epsilon}\left(  X_{\epsilon
}\left(  r\right)  \right)  dr\right)  ^{2}\leq4\left\vert X\left(  0\right)
\right\vert ^{4}+4\left(  \int_{0}^{t}\left\vert \sigma_{\epsilon}\left(
r\right)  \right\vert ^{2}dr\right)  ^{2}\medskip\\
\quad+16\displaystyle\left(  \int_{0}^{t}\left\langle X_{\epsilon
}(r),b_{\epsilon}\left(  r\right)  \right\rangle dr\right)  ^{2}+16\left(
\int_{0}^{t}\left\langle X_{\epsilon}(r),\sigma_{\epsilon}\left(  r\right)
\right\rangle dW(r)\right)  ^{2}.
\end{array}
\label{Ito for approx 2}%
\end{equation}
By Doob inequality we see that%
\[%
\begin{array}
[c]{l}%
\mathbb{E}\sup_{s\in\left[  0,t\right]  }\displaystyle\left\vert \int_{0}%
^{s}\left\langle X_{\epsilon}\left(  r\right)  ,\sigma_{\epsilon}\left(
r\right)  \right\rangle dW\left(  r\right)  \right\vert ^{2}\leq
4\mathbb{E}\left[  \displaystyle\int_{0}^{t}\left\vert \left\langle
X_{\epsilon}(r),\sigma_{\epsilon}\left(  r\right)  \right\rangle \right\vert
^{2}dr\right]  \medskip\\
\leq4\mathbb{E}\displaystyle\int_{0}^{t}\left\vert X_{\epsilon}\left(
r\right)  \right\vert ^{2}\left[  6\ell^{2}\left(  \left\vert X_{\epsilon
}\left(  r\right)  \right\vert ^{2}+\left\vert Y_{\epsilon}\left(  r\right)
\right\vert ^{2}+\left\vert Z_{\epsilon}\left(  r\right)  \right\vert
^{2}\right)  +2\left\vert \sigma_{\epsilon}\left(  0\right)  \right\vert
^{2}\right]  dr\medskip\\
\leq24\ell^{2}\mathbb{E}\displaystyle\left[  \sup_{s\in\left[  0,t\right]
}\left\vert X_{\epsilon}\left(  s\right)  \right\vert ^{2}\int_{0}%
^{t}\left\vert X_{\epsilon}\left(  r\right)  \right\vert ^{2}dr\right]
+24\ell^{2}\mathbb{E}\left[  \int_{0}^{t}\left\vert X_{\epsilon}\left(
r\right)  \right\vert ^{2}\left\vert Y_{\epsilon}\left(  r\right)  \right\vert
^{2}dr\right]  \medskip\\
\quad+24\ell^{2}\mathbb{E}\displaystyle\left[  \int_{0}^{t}\left\vert
X_{\epsilon}\left(  r\right)  \right\vert ^{2}\left\vert Z_{\epsilon}\left(
r\right)  \right\vert ^{2}dr\right]  +8\kappa^{2}\mathbb{E}\left[  \int
_{0}^{t}\left\vert X_{\epsilon}\left(  r\right)  \right\vert ^{2}dr\right]
\medskip\\
\leq\displaystyle\frac{1}{4}\mathbb{E}\sup_{s\in\left[  0,t\right]
}\left\vert X_{\epsilon}\left(  s\right)  \right\vert ^{4}+4\kappa^{2}+\left(
24^{2}\ell^{4}+4\kappa^{2}\right)  \mathbb{E}\left(  \int_{0}^{t}\left\vert
X_{\epsilon}\left(  r\right)  \right\vert ^{2}dr\right)  ^{2}\medskip\\
\quad+24\ell^{2}\delta\mathbb{E}\displaystyle\left[  \int_{0}^{t}\left\vert
X_{\epsilon}\left(  r\right)  \right\vert ^{2}\left(  \int_{-\delta}%
^{0}\left\vert \xi\left(  u\right)  \right\vert ^{2}du\right)  dr\right]
+24\ell^{2}\delta\mathbb{E}\left[  \int_{0}^{t}\left\vert X_{\epsilon}\left(
r\right)  \right\vert ^{2}\left(  \int_{0}^{r}\left\vert X_{\epsilon}\left(
u\right)  \right\vert ^{2}du\right)  dr\right]  \medskip\\
\quad+24\ell^{2}\mathbb{E}\displaystyle\left[  \int_{0}^{t}\left\vert
X_{\epsilon}\left(  r\right)  \right\vert ^{2}\left\vert X_{\epsilon}\left(
r-\delta\right)  \right\vert ^{2}dr\right]  \medskip\\
\leq\displaystyle\frac{1}{2}\mathbb{E}\sup_{s\in\left[  0,t\right]
}\left\vert X_{\epsilon}\left(  s\right)  \right\vert ^{4}+C_{1}%
+C_{2}\mathbb{E}\int_{0}^{t}\left\vert X_{\epsilon}\left(  r\right)
\right\vert ^{4}dr+C_{3}\mathbb{E}\int_{-\delta}^{0}\left\vert \xi\left(
r\right)  \right\vert ^{4}dr.
\end{array}
\]
From (\ref{ineq Yosida 2}) we easily get%
\[
\mathbb{E}\sup_{s\in\left[  0,t\right]  }\left\vert X_{\epsilon}\left(
s\right)  \right\vert ^{4}+\displaystyle\mathbb{E}\left[  \int_{0}^{t}%
\varphi_{\epsilon}\left(  X_{\epsilon}\left(  r\right)  \right)  dr\right]
^{2}\leq C_{4}\mathbb{E}\left[  1+\left\vert \xi\left(  0\right)  \right\vert
^{4}+\int_{-\delta}^{0}\left\vert \xi\left(  r\right)  \right\vert
^{4}dr\right]  +C_{5}\int_{0}^{t}\mathbb{E}\sup_{s\in\left[  0,r\right]
}\left\vert X_{\epsilon}\left(  s\right)  \right\vert ^{4}dr
\]
and by Gronwall's inequality we obtain the conclusion (\ref{ineq bound}).

Also from (\ref{Ito for approx}) it can be deduced, by similar computation,
inequality (\ref{ineq bound 2}).$\medskip$

\noindent\textrm{B.} \textit{Boundedness of }$\nabla\varphi_{\epsilon}\left(
X_{\epsilon}\left(  r\right)  \right)  $

Let $u_{0}\in\mathrm{Int}\left(  \mathrm{Dom}\left(  \varphi\right)  \right)
$ and we recall (\ref{ineq Yosida}-$vii$). It\^{o}'s formula yields%
\begin{equation}%
\begin{array}
[c]{l}%
2r_{0}\displaystyle\int_{0}^{t}\left\vert \nabla\varphi_{\epsilon}\left(
X_{\epsilon}\left(  r\right)  \right)  \right\vert dr-2M_{0}t\leq\left\vert
X\left(  0\right)  -u_{0}\right\vert ^{2}+\int_{0}^{t}\left\vert
\sigma_{\epsilon}\left(  r\right)  \right\vert ^{2}dr\medskip\\
\quad+2\displaystyle\int_{0}^{t}\left\langle X_{\epsilon}(r)-u_{0}%
,b_{\epsilon}\left(  r\right)  \right\rangle dr+2\int_{0}^{t}\left\langle
X_{\epsilon}(r)-u_{0},\sigma_{\epsilon}\left(  r\right)  \right\rangle dW(r)
\end{array}
\label{Ito for approx 3}%
\end{equation}
hence, by the isometry of the stochastic integral,%
\begin{equation}
\mathbb{E}\left(  \int_{0}^{T}\left\vert \nabla\varphi_{\epsilon}\left(
X_{\epsilon}\left(  r\right)  \right)  \right\vert dr\right)  ^{2}\leq
C~\big[1+\mathbb{E}\left\Vert \xi\right\Vert _{\mathcal{C}}^{2}\big].
\label{ineq bound 3}%
\end{equation}
It is immediately that%
\begin{equation}%
\begin{array}
[c]{l}%
\mathbb{E}\left(  \left\Vert K_{\epsilon}\right\Vert _{\mathrm{BV}\left(
\left[  -\delta,T\right]  \right)  }^{2}\right)  \leq C~\big[1+\mathbb{E}%
\left\Vert \xi\right\Vert _{\mathcal{C}}^{4}\big]\quad\text{and}\medskip\\
\mathbb{E}\left(  \left\Vert K_{\epsilon}\right\Vert _{\mathrm{BV}\left(
\left[  -\delta,T\right]  \right)  }\right)  \leq C~\big[1+\mathbb{E}%
\left\Vert \xi\right\Vert _{\mathcal{C}}^{2}\big].
\end{array}
\label{ineq bound 4}%
\end{equation}
Next we shall prove that%
\begin{equation}%
\begin{array}
[c]{l}%
\displaystyle\mathbb{E}\sup_{r\in\left[  0,T\right]  }\left\vert \nabla
\varphi_{\epsilon}\left(  X_{\epsilon}\left(  r\right)  \right)  \right\vert
^{4}\leq\frac{1}{\epsilon^{3}\sqrt{\epsilon}}C\mathbb{E}\left[  1+\left\vert
\xi\left(  0\right)  \right\vert ^{4}+\varphi^{2}\left(  \xi\left(  0\right)
\right)  +\int_{-\delta}^{0}\left\vert \xi\left(  r\right)  \right\vert
^{4}dr\right]  \medskip\\
\leq\displaystyle\frac{1}{\epsilon^{3}\sqrt{\epsilon}}C\left[  1+\mathbb{E}%
\varphi^{2}\left(  \xi\left(  0\right)  \right)  +\mathbb{E}\left\Vert
\xi\right\Vert _{\mathcal{C}}^{4}\right]  =:\frac{1}{\epsilon^{3}%
\sqrt{\epsilon}}C\,\Gamma_{2}~.
\end{array}
\label{ineq bound 5}%
\end{equation}
We cannot apply the It\^{o}'s formula for $\varphi_{\epsilon}^{2}\left(
X_{\epsilon}\left(  t\right)  \right)  $ but, since $\varphi_{\epsilon}$ is of
class $C^{1}$, we can apply Remark 2.34 from \cite{pa-ra/12} and we obtain%
\[%
\begin{array}
[c]{l}%
\varphi_{\epsilon}^{2}\left(  X_{\epsilon}\left(  t\right)  \right)
+2\displaystyle\int_{0}^{t}\varphi_{\epsilon}\left(  X_{\epsilon}\left(
r\right)  \right)  \left\vert \nabla\varphi_{\epsilon}\left(  X_{\epsilon
}\left(  r\right)  \right)  \right\vert ^{2}dr\medskip\\
\leq\varphi_{\epsilon}^{2}\left(  X_{\epsilon}\left(  0\right)  \right)
+2\displaystyle\int_{0}^{t}\varphi_{\epsilon}\left(  X_{\epsilon}\left(
r\right)  \right)  \left\langle \nabla\varphi_{\epsilon}\left(  X_{\epsilon
}\left(  r\right)  \right)  ,b_{\epsilon}\left(  r\right)  \right\rangle
dr+\int_{0}^{t}\left\vert \nabla\varphi_{\epsilon}\left(  X_{\epsilon}\left(
r\right)  \right)  \right\vert ^{2}\left\vert \sigma_{\epsilon}\left(
r\right)  \right\vert ^{2}dr\medskip\\
\quad+\displaystyle\frac{1}{\epsilon}\int_{0}^{t}\varphi_{\epsilon}\left(
X_{\epsilon}\left(  r\right)  \right)  \left\vert \sigma_{\epsilon}\left(
r\right)  \right\vert ^{2}dr+2\int_{0}^{t}\varphi_{\epsilon}\left(
X_{\epsilon}\left(  r\right)  \right)  \left\langle \nabla\varphi_{\epsilon
}\left(  X_{\epsilon}\left(  r\right)  \right)  ,\sigma_{\epsilon}\left(
r\right)  \right\rangle dW\left(  r\right)  ,
\end{array}
\]
where $b_{\epsilon}$ and $\sigma_{\epsilon}$ are defined by
(\ref{def of b,sigma eps}).

From Doob's inequality and (\ref{ineq Yosida 2}-$ix$) we deduce that%
\[%
\begin{array}
[c]{l}%
2\mathbb{E}\sup_{s\in\left[  0,t\right]  }\displaystyle\left\vert \int_{0}%
^{s}\varphi_{\epsilon}\left(  X_{\epsilon}\left(  r\right)  \right)
\left\langle \nabla\varphi_{\epsilon}\left(  X_{\epsilon}\left(  r\right)
\right)  ,\sigma_{\epsilon}\left(  r\right)  \right\rangle dW\left(  r\right)
\right\vert \medskip\\
\leq6\mathbb{E}\displaystyle\left[  \int_{0}^{t}\varphi_{\epsilon}^{2}\left(
X_{\epsilon}\left(  r\right)  \right)  \left\vert \nabla\varphi_{\epsilon
}\left(  X_{\epsilon}\left(  r\right)  \right)  \right\vert ^{2}\left\vert
\sigma_{\epsilon}\left(  r\right)  \right\vert ^{2}dr\right]  ^{1/2}\medskip\\
\leq6\mathbb{E}\displaystyle\left[  \sup_{s\in\left[  0,t\right]  }%
\varphi_{\epsilon}\left(  X_{\epsilon}\left(  r\right)  \right)  \left(
\int_{0}^{t}\frac{2}{\epsilon}\varphi_{\epsilon}\left(  X_{\epsilon}\left(
r\right)  \right)  \left\vert \sigma_{\epsilon}\left(  r\right)  \right\vert
^{2}dr\right)  ^{1/2}\right]  \medskip\\
\leq\displaystyle\frac{1}{2}\mathbb{E}\sup_{s\in\left[  0,t\right]  }%
\varphi_{\epsilon}^{2}\left(  X_{\epsilon}\left(  r\right)  \right)
+\frac{36}{\epsilon}\mathbb{E}\left[  \int_{0}^{t}\varphi_{\epsilon}\left(
X_{\epsilon}\left(  r\right)  \right)  \left\vert \sigma_{\epsilon}\left(
r\right)  \right\vert ^{2}dr\right]  .
\end{array}
\]
Hence, using (\ref{ineq Yosida}-$ii$) and (\ref{ineq Yosida 2}), we get%
\begin{equation}%
\begin{array}
[c]{l}%
\displaystyle\frac{1}{2}\mathbb{E}\sup_{r\in\left[  0,t\right]  }%
\varphi_{\epsilon}^{2}\left(  X_{\epsilon}\left(  r\right)  \right)
+2\int_{0}^{t}\varphi_{\epsilon}\left(  X_{\epsilon}\left(  r\right)  \right)
\left\vert \nabla\varphi_{\epsilon}\left(  X_{\epsilon}\left(  r\right)
\right)  \right\vert ^{2}dr\medskip\\
\leq\mathbb{E}\varphi^{2}\left(  X_{\epsilon}\left(  0\right)  \right)
+\displaystyle\frac{2}{\epsilon}\int_{0}^{t}\varphi_{\epsilon}\left(
X_{\epsilon}\left(  r\right)  \right)  \left\vert X_{\epsilon}\left(
r\right)  \right\vert \left\vert b_{\epsilon}\left(  r\right)  \right\vert
dr+\frac{39}{\epsilon}\int_{0}^{t}\varphi_{\epsilon}\left(  X_{\epsilon
}\left(  r\right)  \right)  \left\vert \sigma_{\epsilon}\left(  r\right)
\right\vert ^{2}dr\medskip\\
=\mathbb{E}\varphi^{2}\left(  X_{\epsilon}\left(  0\right)  \right)
+\displaystyle\int_{0}^{t}\varphi_{\epsilon}\left(  X_{\epsilon}\left(
r\right)  \right)  A_{\epsilon}\left(  X_{\epsilon}\left(  r\right)  \right)
dr,
\end{array}
\label{Ito for approx 4}%
\end{equation}
where%
\[
A_{\epsilon}\left(  X_{\epsilon}\left(  r\right)  \right)  :=\frac{2}%
{\epsilon}\left\vert X_{\epsilon}\left(  r\right)  \right\vert \left\vert
b_{\epsilon}\left(  r\right)  \right\vert +\frac{39}{\epsilon}\left\vert
\sigma_{\epsilon}\left(  r\right)  \right\vert ^{2}.
\]
Now, using (\ref{ineq Yosida 2}-$ix$), Young's inequality and the convexity of
the function $\alpha\left(  x\right)  =x^{3/2}$, it follows%
\[%
\begin{array}
[c]{l}%
\varphi_{\epsilon}\left(  X_{\epsilon}\left(  r\right)  \right)  A_{\epsilon
}\left(  X_{\epsilon}\left(  r\right)  \right)  =\varphi_{\epsilon}%
^{1/3}\left(  X_{\epsilon}\left(  r\right)  \right)  \varphi_{\epsilon}%
^{2/3}\left(  X_{\epsilon}\left(  r\right)  \right)  A_{\epsilon}\left(
X_{\epsilon}\left(  r\right)  \right)  \medskip\\
\leq\varphi_{\epsilon}^{1/3}\left(  X_{\epsilon}\left(  r\right)  \right)
\left\vert \nabla\varphi_{\epsilon}\left(  X_{\epsilon}\left(  r\right)
\right)  \right\vert ^{2/3}\left\vert X_{\epsilon}\left(  r\right)
\right\vert ^{2/3}A_{\epsilon}\left(  X_{\epsilon}\left(  r\right)  \right)
\medskip\\
=\left[  3\varphi_{\epsilon}\left(  X_{\epsilon}\left(  r\right)  \right)
\left\vert \nabla\varphi_{\epsilon}\left(  X_{\epsilon}\left(  r\right)
\right)  \right\vert ^{2}\right]  ^{1/3}\cdot\frac{1}{\sqrt[3]{3}}\left\vert
X_{\epsilon}\left(  r\right)  \right\vert ^{2/3}A_{\epsilon}\left(
X_{\epsilon}\left(  r\right)  \right)  \medskip\\
\leq\varphi_{\epsilon}\left(  X_{\epsilon}\left(  r\right)  \right)
\left\vert \nabla\varphi_{\epsilon}\left(  X_{\epsilon}\left(  r\right)
\right)  \right\vert ^{2}+\frac{2}{3\sqrt{3}}\left\vert X_{\epsilon}\left(
r\right)  \right\vert A_{\epsilon}^{3/2}\left(  X_{\epsilon}\left(  r\right)
\right)  \medskip\\
\leq\varphi_{\epsilon}\left(  X_{\epsilon}\left(  r\right)  \right)
\left\vert \nabla\varphi_{\epsilon}\left(  X_{\epsilon}\left(  r\right)
\right)  \right\vert ^{2}+\frac{C}{\epsilon^{3/2}}\left\vert X_{\epsilon
}\left(  r\right)  \right\vert \left(  1+\left\vert X_{\epsilon}\left(
r\right)  \right\vert ^{2}+\left\vert Y_{\epsilon}\left(  r\right)
\right\vert ^{2}+\left\vert Z_{\epsilon}\left(  r\right)  \right\vert
^{2}\right)  ^{3/2}\medskip\\
\leq\varphi_{\epsilon}\left(  X_{\epsilon}\left(  r\right)  \right)
\left\vert \nabla\varphi_{\epsilon}\left(  X_{\epsilon}\left(  r\right)
\right)  \right\vert ^{2}+\frac{C}{\epsilon^{3/2}}\left\vert X_{\epsilon
}\left(  r\right)  \right\vert \left(  1+\left\vert X_{\epsilon}\left(
r\right)  \right\vert ^{3}+\left\vert Y_{\epsilon}\left(  r\right)
\right\vert ^{3}+\left\vert Z_{\epsilon}\left(  r\right)  \right\vert
^{3}\right)  \medskip\\
\leq\varphi_{\epsilon}\left(  X_{\epsilon}\left(  r\right)  \right)
\left\vert \nabla\varphi_{\epsilon}\left(  X_{\epsilon}\left(  r\right)
\right)  \right\vert ^{2}+\frac{C}{\epsilon^{3/2}}\left(  1+\left\vert
X_{\epsilon}\left(  r\right)  \right\vert ^{4}+\left\vert Y_{\epsilon}\left(
r\right)  \right\vert ^{4}+\left\vert Z_{\epsilon}\left(  r\right)
\right\vert ^{4}\right)
\end{array}
\]
and (\ref{Ito for approx 4}) becomes%
\[%
\begin{array}
[c]{l}%
\displaystyle\frac{1}{2}\mathbb{E}\sup_{r\in\left[  0,t\right]  }%
\varphi_{\epsilon}^{2}\left(  X_{\epsilon}\left(  r\right)  \right)
+\displaystyle\mathbb{E}\int_{0}^{t}\varphi_{\epsilon}\left(  X_{\epsilon
}\left(  r\right)  \right)  \left\vert \nabla\varphi_{\epsilon}\left(
X_{\epsilon}\left(  r\right)  \right)  \right\vert ^{2}dr\medskip\\
\leq\mathbb{E}\varphi^{2}\left(  X_{\epsilon}\left(  0\right)  \right)
+\displaystyle\frac{C}{\epsilon^{3/2}}\mathbb{E}\int_{0}^{t}\left(
1+\left\vert X_{\epsilon}\left(  r\right)  \right\vert ^{4}+\left\vert
Y_{\epsilon}\left(  r\right)  \right\vert ^{4}+\left\vert Z_{\epsilon}\left(
r\right)  \right\vert ^{4}\right)  dr\medskip\\
\leq\mathbb{E}\varphi^{2}\left(  \xi\left(  0\right)  \right)
+\displaystyle\frac{1}{\epsilon^{3/2}}C\mathbb{E}\left(  1+\left\vert
\xi\left(  0\right)  \right\vert ^{4}+\int_{-\delta}^{0}\left\vert \xi\left(
r\right)  \right\vert ^{4}dr\right)  \mathbb{\leq}\frac{1}{\epsilon^{3/2}%
}C\Gamma_{2},
\end{array}
\]
since $Y_{\epsilon}$ and $Z_{\epsilon}$ are defined by (\ref{def of Y,Z eps}).

The conclusion (\ref{ineq bound 5}) follows now using (\ref{ineq Yosida 2}%
-$ix$).$\medskip$

\noindent\textrm{C.} \textit{Cauchy property of the sequence }$\left(
X_{\epsilon},K_{\epsilon}\right)  $

Let $\epsilon,\tau\in(0,1]$. It\^{o}'s formula yields%
\[%
\begin{array}
[c]{l}%
|X_{\epsilon}\left(  t\right)  -X_{\tau}\left(  t\right)  |^{2}%
+2\displaystyle\int_{0}^{t}\left\langle \nabla\varphi_{\epsilon}\left(
X_{\epsilon}\left(  r\right)  \right)  -\nabla\varphi_{\tau}\left(  X_{\tau
}\left(  r\right)  \right)  ,X_{\epsilon}\left(  r\right)  -X_{\tau}\left(
r\right)  \right\rangle dr\medskip\\
=2\displaystyle\int_{0}^{t}\left\langle X_{\epsilon}\left(  r\right)
-X_{\tau}\left(  r\right)  ,b_{\epsilon}\left(  r\right)  -b_{\tau}\left(
r\right)  \right\rangle dr+\int_{0}^{t}\left\vert \sigma_{\epsilon}\left(
r\right)  -\sigma_{\tau}\left(  r\right)  \right\vert ^{2}dr\medskip\\
\quad+2\displaystyle\int_{0}^{t}\left\langle X_{\epsilon}\left(  r\right)
-X_{\tau}\left(  r\right)  ,\sigma_{\epsilon}\left(  r\right)  -\sigma_{\tau
}\left(  r\right)  \right\rangle dW\left(  r\right)  ,\;\forall t\in
\lbrack0,T],
\end{array}
\]
where $b_{\epsilon}$, $b_{\tau}$, $\sigma_{\epsilon}$ and $\sigma_{\tau}$ are
defined by (\ref{def of b,sigma eps}).

From (\ref{ineq Yosida}-$vi$) and Doob's inequality%
\[%
\begin{array}
[c]{l}%
\mathbb{E}\sup_{r\in\left[  0,t\right]  }\left\vert X_{\epsilon}\left(
r\right)  -X_{\tau}\left(  r\right)  \right\vert ^{2}\leq2\left(
\epsilon+\delta\right)  \mathbb{E}\displaystyle\int_{0}^{t}\left\vert
\nabla\varphi_{\epsilon}\left(  X_{\epsilon}\left(  r\right)  \right)
\right\vert \left\vert \nabla\varphi_{\tau}\left(  X_{\tau}\left(  r\right)
\right)  \right\vert dr\medskip\\
\quad+\left(  4\ell+3\ell^{2}\right)  \mathbb{E}\displaystyle\int_{0}%
^{t}\left\vert X_{\epsilon}\left(  r\right)  -X_{\tau}\left(  r\right)
\right\vert ^{2}dr+\left(  \ell+3\ell^{2}\right)  \mathbb{E}\int_{0}%
^{t}\left\vert Y_{\epsilon}\left(  r\right)  -Y_{\tau}\left(  r\right)
\right\vert ^{2}dr\medskip\\
\quad+\left(  \ell+3\ell^{2}\right)  \displaystyle\mathbb{E}\int_{0}%
^{t}\left\vert Z_{\epsilon}\left(  r\right)  -Z_{\tau}\left(  r\right)
\right\vert ^{2}dr+6\mathbb{E}\left[  \int_{0}^{t}\left\vert X_{\epsilon
}\left(  r\right)  -X_{\tau}\left(  r\right)  \right\vert ^{2}\left\vert
\sigma_{\epsilon}\left(  r\right)  -\sigma_{\tau}\left(  r\right)  \right\vert
^{2}dr\right]  ^{1/2}.
\end{array}
\]
But%
\begin{equation}%
\begin{array}
[c]{l}%
\mathbb{E}\displaystyle\int_{0}^{t}\left\vert Y_{\epsilon}\left(  r\right)
-Y_{\tau}\left(  r\right)  \right\vert ^{2}dr=\mathbb{E}\int_{0}^{t}\left\vert
\int_{-\delta}^{0}e^{\lambda u}X_{\epsilon}\left(  r+u\right)  -X_{\tau
}\left(  r+u\right)  du\right\vert ^{2}dr\medskip\\
\leq\delta\mathbb{E}\displaystyle\int_{0}^{t}\left(  \int_{-\delta}%
^{0}\left\vert X_{\epsilon}\left(  r+u\right)  -X_{\tau}\left(  r+u\right)
\right\vert ^{2}du\right)  dr=\delta\mathbb{E}\int_{0}^{t}\left(
\int_{r-\delta}^{r}\left\vert X_{\epsilon}\left(  u\right)  -X_{\tau}\left(
u\right)  \right\vert ^{2}du\right)  dr\medskip\\
\leq\delta t~\mathbb{E}\displaystyle\int_{-\delta}^{t}\left\vert X_{\epsilon
}\left(  u\right)  -X_{\tau}\left(  u\right)  \right\vert ^{2}du=\delta
t~\mathbb{E}\int_{0}^{t}\left\vert X_{\epsilon}\left(  u\right)  -X_{\tau
}\left(  u\right)  \right\vert ^{2}du
\end{array}
\label{Cauchy property of Y}%
\end{equation}
and%
\[%
\begin{array}
[c]{l}%
\mathbb{E}\displaystyle\int_{0}^{t}\left\vert Z_{\epsilon}\left(  r\right)
-Z_{\tau}\left(  r\right)  \right\vert ^{2}dr=\mathbb{E}\int_{0}^{t}\left\vert
X_{\epsilon}\left(  r-\delta\right)  -X_{\tau}\left(  r-\delta\right)
\right\vert ^{2}dr=\mathbb{E}\int_{-\delta}^{t-\delta}\left\vert X_{\epsilon
}\left(  r\right)  -X_{\tau}\left(  r\right)  \right\vert ^{2}dr\medskip\\
\leq\mathbb{E}\displaystyle\int_{-\delta}^{t}\left\vert X_{\epsilon}\left(
r\right)  -X_{\tau}\left(  r\right)  \right\vert ^{2}dr=\mathbb{E}\int_{0}%
^{t}\left\vert X_{\epsilon}\left(  r\right)  -X_{\tau}\left(  r\right)
\right\vert ^{2}dr,
\end{array}
\]
therefore%
\[%
\begin{array}
[c]{l}%
\mathbb{E}\sup_{r\in\left[  0,t\right]  }\left\vert X_{\epsilon}\left(
r\right)  -X_{\tau}\left(  r\right)  \right\vert ^{2}\leq2\left(
\epsilon+\delta\right)  \mathbb{E}\displaystyle\int_{0}^{t}\left\vert
\nabla\varphi_{\epsilon}\left(  X_{\epsilon}\left(  r\right)  \right)
\right\vert \left\vert \nabla\varphi_{\tau}\left(  X_{\tau}\left(  r\right)
\right)  \right\vert dr\medskip\\
\quad+C\mathbb{E}\displaystyle\int_{0}^{t}\left\vert X_{\epsilon}\left(
r\right)  -X_{\tau}\left(  r\right)  \right\vert ^{2}dr+C\mathbb{E}\left[
\sup_{s\in\left[  0,t\right]  }\left\vert X_{\epsilon}\left(  s\right)
-X_{\tau}\left(  s\right)  \right\vert \left(  \int_{0}^{t}\left\vert
X_{\epsilon}\left(  r\right)  -X_{\tau}\left(  r\right)  \right\vert
^{2}dr\right)  ^{1/2}\right]  \medskip\\
\leq2\left(  \epsilon+\delta\right)  \mathbb{E}\displaystyle\int_{0}%
^{t}\left\vert \nabla\varphi_{\epsilon}\left(  X_{\epsilon}\left(  r\right)
\right)  \right\vert \left\vert \nabla\varphi_{\tau}\left(  X_{\tau}\left(
r\right)  \right)  \right\vert dr+C\mathbb{E}\displaystyle\int_{0}%
^{t}\left\vert X_{\epsilon}\left(  r\right)  -X_{\tau}\left(  r\right)
\right\vert ^{2}dr+\medskip\\
\quad+\displaystyle\frac{1}{2}\mathbb{E}\sup_{s\in\left[  0,t\right]
}\left\vert X_{\epsilon}\left(  s\right)  -X_{\tau}\left(  s\right)
\right\vert ^{2}.
\end{array}
\]
On the other hand, using two times H\"{o}lder's inequality,
(\ref{ineq bound 5}) and (\ref{ineq bound 3}),%
\[%
\begin{array}
[c]{l}%
2\left(  \epsilon+\delta\right)  \mathbb{E}\displaystyle\int_{0}^{t}\left\vert
\nabla\varphi_{\epsilon}\left(  X_{\epsilon}\left(  r\right)  \right)
\right\vert \left\vert \nabla\varphi_{\tau}\left(  X_{\tau}\left(  r\right)
\right)  \right\vert dr\medskip\\
\leq2\epsilon\mathbb{E}\displaystyle\left[  \sup_{s\in\left[  0,t\right]
}\left\vert \nabla\varphi_{\epsilon}\left(  X_{\epsilon}\left(  s\right)
\right)  \right\vert \int_{0}^{t}\left\vert \nabla\varphi_{\tau}\left(
X_{\tau}\left(  r\right)  \right)  \right\vert dr\right]  +2\delta
\mathbb{E}\left\vert \sup_{s\in\left[  0,t\right]  }\left\vert \nabla
\varphi_{\tau}\left(  X_{\tau}\left(  s\right)  \right)  \right\vert \int
_{0}^{t}\left\vert \nabla\varphi_{\epsilon}\left(  X_{\epsilon}\left(
r\right)  \right)  \right\vert dr\right\vert \medskip\\
\leq2\epsilon\displaystyle\left[  \mathbb{E}\sup_{s\in\left[  0,t\right]
}\left\vert \nabla\varphi_{\epsilon}\left(  X_{\epsilon}\left(  s\right)
\right)  \right\vert ^{4}\right]  ^{1/4}\left[  \mathbb{E}\left(  \int_{0}%
^{t}\left\vert \nabla\varphi_{\tau}\left(  X_{\tau}\left(  r\right)  \right)
\right\vert dr\right)  ^{2}\right]  ^{1/2}\medskip\\
\quad+2\delta\displaystyle\left[  \mathbb{E}\sup_{s\in\left[  0,t\right]
}\left\vert \nabla\varphi_{\delta}\left(  X_{\delta}\left(  s\right)  \right)
\right\vert ^{4}\right]  ^{1/4}\left[  \mathbb{E}\left(  \int_{0}%
^{t}\left\vert \nabla\varphi_{\delta}\left(  X_{\delta}\left(  r\right)
\right)  \right\vert dr\right)  ^{2}\right]  ^{1/2}\medskip\\
\leq C\big(\epsilon^{1/8}+\delta^{1/8}\big)\Gamma_{2}^{1/4},
\end{array}
\]
hence%
\[
\frac{1}{2}\mathbb{E}\sup_{s\in\left[  0,t\right]  }|X_{\epsilon}\left(
s\right)  -X_{\tau}\left(  s\right)  |^{2}\leq C\big(\epsilon^{1/8}%
+\delta^{1/8}\big)\Gamma_{2}^{1/4}+C\int_{0}^{t}\mathbb{E}\sup_{s\in\left[
0,r\right]  }\left\vert X_{\epsilon}\left(  s\right)  -X_{\tau}\left(
s\right)  \right\vert ^{2}dr.
\]
Gronwall's inequality yields%
\begin{equation}
\mathbb{E}\sup_{s\in\left[  0,t\right]  }|X_{\epsilon}\left(  s\right)
-X_{\tau}\left(  s\right)  |^{2}\leq C\big(\epsilon^{1/8}+\delta
^{1/8}\big)\Gamma_{2}^{1/4}~. \label{Cauchy property of X}%
\end{equation}
Using equation (\ref{SVI delayed approxim}) we can deduce the following
inequality quite easily:%
\[
\mathbb{E}\sup_{s\in\left[  0,t\right]  }|K_{\epsilon}\left(  s\right)
-K_{\tau}\left(  s\right)  |^{2}\leq C\big(\epsilon^{1/8}+\delta
^{1/8}\big)\Gamma_{2}^{1/4}~.
\]
\noindent\textrm{D.} \textit{Passing to the limit}

Taking into account the Cauchy property we deduce that there exist
$\lim_{\epsilon\rightarrow0}X_{\epsilon}=X$ and $\lim_{\epsilon\rightarrow
0}K_{\epsilon}=K$ with $X,K\in L_{ad}^{2}\left(  \Omega;C\left(  \left[
-\delta,T\right]  \right)  \right)  $. Moreover, from (\ref{ineq bound 4}) we
see that there exists $\epsilon_{n}\rightarrow0$ such that%
\[
K_{\epsilon_{n}}\rightharpoonup K\text{ weakly in }L_{ad}^{2}\left(
\Omega;\mathrm{BV}\left(  \left[  -\delta,T\right]  \right)  \right)
\]
and%
\[
\mathbb{E}\left(  \left\Vert K\right\Vert _{\mathrm{BV}\left(  \left[
-\delta,T\right]  \right)  }^{2}\right)  \leq C~\big[1+\mathbb{E}\left\Vert
\xi\right\Vert _{\mathcal{C}}^{4}\big].
\]
Passing to the limit in the approximate equation we obtain that $\left(
X,K\right)  $ satisfies (\ref{def sol of delay SVI 2}-$iv$).

From (\ref{ineq bound 5}) we have
\[
\mathbb{E}\sup_{s\in\left[  0,T\right]  }\left\vert X\left(  s\right)
-J_{\epsilon}\left(  X_{\epsilon}\left(  s\right)  \right)  \right\vert
^{4}\leq C\sqrt{\epsilon}\Gamma_{2}~.
\]
Moreover for $\forall z\in\mathbb{R}^{d}$,%
\[
\int_{t}^{\hat{t}}\left\langle \nabla\varphi_{\epsilon}\left(  X_{\epsilon
}\left(  r\right)  \right)  ,z-J_{\epsilon}\left(  X_{\epsilon}\left(
r\right)  \right)  \right\rangle dr+\int_{t}^{\hat{t}}\varphi\left(
J_{\epsilon}\left(  X_{\epsilon}\left(  r\right)  \right)  \right)
dr\leq\left(  \hat{t}-t\right)  \varphi\left(  z\right)  ,
\]
since we have (\ref{ineq Yosida}-$iii$). Passing to the limit we obtain
(\ref{def sol of delay SVI 2}-$vi$).\hfill

\section{The optimal problem}

The aim of this section is to prove that the value function satisfies the
dynamic programming principle and is a viscosity solution of a partial
differential equation of Hamilton-Jacobi-Bellman (HJB) type.

\subsection{Dynamic programming principle}

Let $\left(  s,\xi\right)  \in\lbrack0,T)\times\mathcal{C}\big(\left[
-\delta,0\right]  ;\overline{\mathrm{Dom}\left(  \varphi\right)  }\,\big)$ be
arbitrary but fixed, $\mathrm{U}\subset\mathbb{R}^{m}$ be a given compact set
of admissible control values and $u:\Omega\times\left[  s,T\right]
\rightarrow\mathrm{U}$ be the control process. As in \cite{yo-zh/99}, we
define the class $\mathcal{U}\left[  s,T\right]  $ of admissible control
strategies as the set of five-tuples $\left(  \Omega,\mathcal{F}%
,\mathbb{P},W,u\right)  $ such that: $\left(  \Omega,\mathcal{F}%
,\{\mathcal{F}_{t}^{s}\}_{t\geq s},\mathbb{P}\right)  $ is a complete
probability space; $\left\{  W\left(  t\right)  \right\}  _{t\geq s}$ is a
$n$-dimensional standard Brownian motion with $W\left(  s\right)  =0$ and
$\{\mathcal{F}_{t}^{s}\}_{t\geq s}$ is generated by the Brownian motion
augmented by the $\mathbb{P}$-null set in $\mathcal{F};$ the control process
$u:\Omega\times\left[  s,T\right]  \rightarrow\mathrm{U}$ is an $\{\mathcal{F}%
_{t}^{s}\}_{t\geq s}$-adapted process and satisfied that $\mathbb{E}%
\big[\int_{s}^{T}\left\vert f\left(  t,X(t),Y(t),u(t)\right)  \right\vert
dt+\left\vert h\left(  X(T),Y(T)\right)  \right\vert \big]<\infty$; equation
(\ref{SVI delayed}) admits a unique strong solution provided $\left(
s,\xi,u\right)  $.

We consider the following stochastic controlled system%
\begin{equation}
\left\{
\begin{array}
[c]{r}%
dX(t)+\partial\varphi\left(  X(t)\right)  dt\ni b\left(
t,X(t),Y(t),Z(t),u(t)\right)  dt+\sigma\left(  t,X(t),Y(t),Z(t),u(t)\right)
dW(t),\medskip\\
t\in(s,T],\medskip\\
\multicolumn{1}{l}{X(t)=\xi\left(  t-s\right)  ,\;t\in\left[  s-\delta
,s\right]  ,}%
\end{array}
\right.  \label{SVI delayed 5}%
\end{equation}
where%
\begin{equation}
Y(t):=\int_{-\delta}^{0}e^{\lambda r}X(t+r)dr,\quad Z(t):=X(t-\delta),
\label{def of Y,Z 2}%
\end{equation}
together with the cost functional%
\begin{equation}
J(s,\xi;u)=\mathbb{E}\big[\int_{s}^{T}f\left(  t,X^{s,\xi,u}(t),Y^{s,\xi
,u}(t),u(t)\right)  dt+h\left(  X^{s,\xi,u}(T),Y^{s,\xi,u}(T)\right)  \big].
\label{funct cost 2}%
\end{equation}
We define the associated value function as the infimum among all
$u\in\mathcal{U}\left[  s,T\right]  :$%
\begin{equation}
V\left(  s,\xi\right)  =\inf_{u\in\mathcal{U}\left[  s,T\right]  }J\left(
s,\xi;u\right)  ,\;\left(  s,\xi\right)  \in\lbrack0,T)\times\mathcal{C}%
\big(\left[  -\delta,0\right]  ;\overline{\mathrm{Dom}\left(  \varphi\right)
}\,\big). \label{funct value 2}%
\end{equation}

\begin{definition}
We say that the value function satisfies the dynamic programming principle
(DPP for short) if, for every $\left(  s,\xi\right)  \in\lbrack0,T)\times
\mathcal{C}\big(\left[  -\delta,0\right]  ;$ $\overline{\mathrm{Dom}\left(
\varphi\right)  }\,\big)$, it holds that%
\begin{equation}
V\left(  s,\xi\right)  =\inf_{u\in\mathcal{U}\left[  s,T\right]  }%
\mathbb{E}\big[\int_{s}^{\theta}f\left(  t,X^{s,\xi,u}(t),Y^{s,\xi
,u}(t),u(t)\right)  dt+V\left(  \theta,X^{s,\xi,u}\left(  \theta\right)
\right)  \big], \label{DPP}%
\end{equation}
for every stopping time $\theta\in\left[  s,T\right]  .$
\end{definition}

As it can be seen in the previous section, the following three assumptions
will be needed to ensure the existence of a solution $X^{s,\xi,u}$ for
(\ref{SVI delayed 5}):

\begin{itemize}
\item[$\mathrm{(H}_{1}\mathrm{)}$] The function $\varphi:\mathbb{R}%
^{d}\rightarrow(-\infty,+\infty]$ is convex and l.s.c. such that
$\mathrm{Int}\left(  \mathrm{Dom}\left(  \varphi\right)  \right)
\neq\emptyset$ and $\varphi(x)\geq\varphi\left(  0\right)  =0,\;\forall
\,x\in\mathbb{R}^{d}.$

\item[$\mathrm{(H}_{2}\mathrm{)}$] The initial path $\xi$ is $\mathcal{F}%
_{s}^{s}$-measurable such that%
\[
\xi\in L^{2}\big(\Omega;\mathcal{C}\big(\left[  -\delta,0\right]
;\overline{\mathrm{Dom}\left(  \varphi\right)  }\,\big)\big),\quad
\text{and}\quad\varphi\left(  \xi\left(  0\right)  \right)  \in L^{1}\left(
\Omega;\mathbb{R}^{d}\right)  .
\]

\item[$\mathrm{(H}_{3}\mathrm{)}$] The functions $b:\left[  0,T\right]
\times\mathbb{R}^{3d}\times\mathrm{U}\rightarrow\mathbb{R}^{d}$ and
$\sigma:\left[  0,T\right]  \times\mathbb{R}^{3d}\times\mathrm{U}%
\rightarrow\mathbb{R}^{d\times n}$ are continuous and there exist $\ell
,\kappa>0$ such that for all $t\in\lbrack0,T]$, $u\in U$ and $x,y,z,x^{\prime
},y^{\prime},z^{\prime}\in\mathbb{R}^{d}$,%
\begin{equation}%
\begin{array}
[c]{r}%
\left\vert b\left(  t,x,y,z,u\right)  -b\left(  t,x^{\prime},y^{\prime
},z^{\prime},u\right)  \right\vert +\left\vert \sigma\left(  t,x,y,z,u\right)
-\sigma\left(  t,x^{\prime},y^{\prime},z^{\prime},u\right)  \right\vert
\medskip\\
\leq\ell\left(  \left\vert x-x^{\prime}\right\vert +\left\vert y-y^{\prime
}\right\vert +\left\vert z-z^{\prime}\right\vert \right)  ,\medskip\\
\multicolumn{1}{l}{\left\vert b\left(  t,0,0,0,u\right)  \right\vert
+\left\vert \sigma\left(  t,0,0,0,u\right)  \right\vert \leq\kappa.}%
\end{array}
\label{differ of b, sigma 2}%
\end{equation}

\end{itemize}

\begin{theorem}
\label{main result 2}Under the assumptions $\mathrm{(H}_{1}-\mathrm{H}%
_{3}\mathrm{)}$, for any $\left(  s,\xi\right)  \in\lbrack0,T)\times
\mathcal{C}\big(\left[  -\delta,0\right]  ;\overline{\mathrm{Dom}\left(
\varphi\right)  }\,\big)$ and $u\in\mathcal{U}\left[  s,T\right]  $ there
exists a unique pair of processes $\left(  X,K\right)  =\left(  X^{s,\xi
,u},K^{s,\xi,u}\right)  $ which is the solution of stochastic variational
inequality with delay (\ref{SVI delayed 5}). In addition, for any $q\geq1$,
there exists some constants $C=C\left(  \ell,\kappa,\delta,T,q\right)  >0$ and
$C^{\,\prime}=C^{\,\prime}\left(  \ell,\kappa,\delta,T,q\right)  >0$ such
that, for any $\left(  s,\xi\right)  $, $(s^{\prime},\xi^{\prime})\in
\lbrack0,T)\times\mathcal{C}\big(\left[  -\delta,0\right]  ;\overline
{\mathrm{Dom}\left(  \varphi\right)  }\,\big)$,%
\begin{equation}%
\begin{array}
[c]{r}%
\displaystyle\mathbb{E}\sup_{r\in\left[  s,T\right]  }\left\vert X^{s,\xi
,u}\left(  r\right)  \right\vert ^{2q}+\mathbb{E}\sup_{r\in\left[  s,T\right]
}\left\vert K^{s,\xi,u}\left(  r\right)  \right\vert ^{2q}+\mathbb{E}%
\left\Vert K^{s,\xi,u}\right\Vert _{\mathrm{BV}\left(  \left[  -\delta
,T\right]  \right)  }^{q}\medskip\\
\displaystyle+\mathbb{E}\Big(\int_{s}^{T}\varphi\left(  X^{s,\xi,u}\left(
r\right)  \right)  dr\Big)^{q}\leq C\big[1+\left\Vert \xi\right\Vert
_{\mathcal{C}}^{2q}\big]
\end{array}
\label{bound for sol}%
\end{equation}
and%
\begin{equation}%
\begin{array}
[c]{r}%
\mathbb{E}\sup_{r\in\left[  s\wedge s^{\prime},t\right]  }|X^{s,\xi}\left(
r\right)  -X^{s^{\prime},\xi^{\prime}}\left(  r\right)  |^{2}+\mathbb{E}%
\sup_{r\in\left[  s\wedge s^{\prime},t\right]  }|K^{s,\xi}\left(  r\right)
-K^{s^{\prime},\xi^{\prime}}\left(  r\right)  |^{2}\medskip\\
\leq C^{\,\prime}\Big[\Gamma_{1}+\left\vert s-s^{\prime}\right\vert
\big(1+\left\Vert \xi\right\Vert _{\mathcal{C}}^{2}+||\xi^{\prime
}||_{\mathcal{C}}^{2}\big)\Big],
\end{array}
\label{differ of 2 solutions}%
\end{equation}
where%
\begin{equation}
\Gamma_{1}=||\xi-\xi^{\prime}||_{\mathcal{C}}^{2}+\int_{s^{\prime}-\delta
}^{s^{\prime}}\left\vert \xi^{\prime}\left(  r-s\right)  -\xi^{\prime}\left(
r-s^{\prime}\right)  \right\vert ^{2}dr \label{def gamma 1_2}%
\end{equation}
(see definition (\ref{def gamma 1})).
\end{theorem}

\begin{remark}
\label{main result 2_remark}Using the above estimations and definition
(\ref{def of Y,Z 2}), it is easily to deduce that%
\begin{equation}
\mathbb{E}\sup_{r\in\left[  s,T\right]  }|Y^{s,\xi,u}\left(  r\right)
|^{2q}\leq C\big[1+\left\Vert \xi\right\Vert _{\mathcal{C}}^{2q}%
\big] \label{bound for sol 2}%
\end{equation}
and%
\begin{equation}
\mathbb{E}\sup_{r\in\left[  s\wedge s^{\prime},t\right]  }|Y^{s,\xi}\left(
r\right)  -Y^{s^{\prime},\xi^{\prime}}\left(  r\right)  |^{2}\leq C^{\prime
}\Big[\Gamma_{1}+\left\vert s-s^{\prime}\right\vert \big(1+\left\Vert
\xi\right\Vert _{\mathcal{C}}^{2}+||\xi^{\prime}||_{\mathcal{C}}%
^{2}\big)\Big]. \label{differ of 2 solutions 2}%
\end{equation}

\end{remark}

Under the next assumption the cost functional and the value function will be well-defined.

\begin{itemize}
\item[$\mathrm{(H}_{4}\mathrm{)}$] The functions $f:\left[  0,T\right]
\times\mathbb{R}^{2d}\times\mathrm{U}\rightarrow\mathbb{R}$, $h:\mathbb{R}%
^{2d}\rightarrow\mathbb{R}$ are continuous and there exists $\bar{\kappa}>0$
and $p\geq1$ such that for all $t\in\lbrack0,T]$, $u\in U$ and $x,y\in
\mathbb{R}^{d}$,%
\[
\left\vert f\left(  t,x,y,u\right)  \right\vert +\left\vert h\left(
x,y\right)  \right\vert \leq\bar{\kappa}\left(  1+\left\vert x\right\vert
^{p}+\left\vert y\right\vert ^{p}\right)  .
\]

\end{itemize}

In the sequel we will follows the techniques from \cite{za/12} in order to
give some basic properties of the value function (including the continuity).

\begin{proposition}
\label{properties of V}Let assumptions $\mathrm{(H}_{1}-\mathrm{H}%
_{4}\mathrm{)}$ be satisfied. Then there exists $C>0$ such that%
\begin{equation}
|V\left(  s,\xi\right)  |\leq C~\big[1+\left\Vert \xi\right\Vert
_{\mathcal{C}}^{p}\big],\;\forall\left(  s,\xi\right)  \in\left[  0,T\right]
\times\mathcal{C}\big(\left[  -\delta,0\right]  ;\overline{\mathrm{Dom}\left(
\varphi\right)  }\,\big) \label{V sublinear}%
\end{equation}
and%
\begin{equation}%
\begin{array}
[c]{l}%
\displaystyle|V\left(  s,\xi\right)  -V(s^{\prime},\xi^{\prime})|\leq
C\,\boldsymbol{\mu}_{f,h}\left(  \delta,M\right)  +C\big[1+\left\Vert
\xi\right\Vert _{\mathcal{C}}^{p}+\left\Vert \xi^{\prime}\right\Vert
_{\mathcal{C}}^{p}\big]\cdot\medskip\\
\displaystyle\quad\quad\quad\quad\quad\quad\quad\quad\quad\quad\bigg[\frac
{\Gamma_{1}^{1/2}+\left\vert s-s^{\prime}\right\vert ^{1/2}\big(1+\left\Vert
\xi\right\Vert _{\mathcal{C}}+||\xi^{\prime}||_{\mathcal{C}}\big)}{\delta
}+\frac{1+\left\Vert \xi\right\Vert _{\mathcal{C}}+\left\Vert \xi^{\prime
}\right\Vert _{\mathcal{C}}}{M}\bigg],\medskip\\
\multicolumn{1}{r}{\forall\left(  s,\xi\right)  ,(s^{\prime},\xi^{\prime}%
)\in\left[  0,T\right]  \times\mathcal{C}\big(\left[  -\delta,0\right]
;\overline{\mathrm{Dom}\left(  \varphi\right)  }\,\big)\;,}%
\end{array}
\label{V lipschitz}%
\end{equation}
where $\boldsymbol{\mu}_{f,h}\left(  \delta,M\right)  $ is the module of
continuity of $f$ and $h$,%
\[
\boldsymbol{\mu}_{f,h}\left(  \delta,M\right)  :=\sup
_{\substack{|x|+|x^{\prime}|+|y|+|y^{\prime}|\leq M\\|x-x^{\prime
}|+|y-y^{\prime}|\leq\delta\\\left(  t,u\right)  \in\left[  0,T\right]
\times\mathrm{U}}}\left\{  |f(t,x,y,u)-f(t,x^{\prime},y^{\prime}%
,u)|+|h(x,y)-h(x^{\prime},y^{\prime})|\right\}  ,\;\text{for }\delta,M>0
\]
and $\Gamma_{1}$ is defined in (\ref{def gamma 1_2}).
\end{proposition}

\begin{proof}
Using the definition of the cost functional we see that%
\[%
\begin{array}
[c]{l}%
\left\vert J\left(  s,\xi;u\right)  -J(s^{\prime},\xi^{\prime};u)\right\vert
\medskip\\
\leq\mathbb{E}\Big|\displaystyle\int_{s}^{T}f\left(  r\right)  dr+h\left(
X\left(  T\right)  ,Y\left(  T\right)  \right)  -\int_{s^{\prime}}%
^{T}f^{\prime}\left(  r\right)  dr-h(X^{\prime}\left(  T\right)  ,Y^{\prime
}\left(  T\right)  )\Big|\medskip\\
\leq\mathbb{E}\Big|\displaystyle\int_{s}^{T}f\left(  r\right)  dr-\int
_{s^{\prime}}^{T}f^{\prime}\left(  r\right)  dr\Big|+\mathbb{E}\big|h\left(
X\left(  T\right)  ,Y\left(  T\right)  \right)  -h(X^{\prime}\left(  T\right)
,Y^{\prime}\left(  T\right)  )\big|,
\end{array}
\]
where%
\[
X=X^{s,\xi,u}\,,~Y=Y^{s,\xi,u}\,,~X^{\prime}=X^{s^{\prime},\xi^{\prime}%
,u}\,,~Y^{\prime}=Y^{s^{\prime},\xi^{\prime},u}%
\]
and%
\[
f\left(  r\right)  =f\left(  r,X\left(  r\right)  ,Y\left(  r\right)
,u\left(  r\right)  \right)  \quad\text{and}\quad f^{\prime}\left(  r\right)
=f(r,X^{\prime}\left(  r\right)  ,Y^{\prime}\left(  r\right)  ,u\left(
r\right)  ).
\]
Next, we let $0\leq s^{\prime}\leq s\leq T.$ We have%
\[%
\begin{array}
[c]{l}%
\displaystyle\left\vert J\left(  s,\xi;u\right)  -J(s^{\prime},\xi^{\prime
};u)\right\vert \leq\mathbb{E}\int_{s}^{T}|f\left(  r\right)  -f^{\prime
}\left(  r\right)  |dr+\mathbb{E}\int_{s^{\prime}}^{s}|f^{\prime}\left(
r\right)  |dr\medskip\\
\quad+\mathbb{E}|h\left(  X\left(  T\right)  ,Y\left(  T\right)  \right)
-h(X^{\prime}\left(  T\right)  ,Y^{\prime}\left(  T\right)  )|\medskip\\
\displaystyle\leq\int_{s}^{T}\mathbb{E}\left[  |f\left(  r\right)  -f^{\prime
}\left(  r\right)  |\left(  \mathbf{1}_{A_{1}\left(  r\right)  }%
+\mathbf{1}_{A_{2}\left(  r\right)  }+\mathbf{1}_{A_{3}\left(  r\right)
}\right)  \right]  dr+\int_{s^{\prime}}^{s}\mathbb{E}|f^{\prime}\left(
r\right)  |dr\medskip\\
\displaystyle\quad+\mathbb{E}\left[  |h\left(  X\left(  T\right)  ,Y\left(
T\right)  \right)  -h(X^{\prime}\left(  T\right)  ,Y^{\prime}\left(  T\right)
)|\left(  \mathbf{1}_{A_{1}\left(  T\right)  }+\mathbf{1}_{A_{2}\left(
T\right)  }+\mathbf{1}_{A_{3}\left(  T\right)  }\right)  \right]  ,
\end{array}
\]
where%
\[%
\begin{array}
[c]{l}%
A_{1}\left(  r\right)  :=\left\{  \omega:|X\left(  r\right)  |+|X^{\prime
}\left(  r\right)  |+|Y\left(  r\right)  |+|Y^{\prime}\left(  r\right)  |\leq
M,\;|X\left(  r\right)  -X^{\prime}\left(  r\right)  |+|Y\left(  r\right)
-Y^{\prime}\left(  r\right)  |\leq\delta\right\}  ,\medskip\\
A_{2}\left(  r\right)  :=\left\{  \omega:|X\left(  r\right)  |+|X^{\prime
}\left(  r\right)  |+|Y\left(  r\right)  |+|Y^{\prime}\left(  r\right)  |\leq
M,\;|X\left(  r\right)  -X^{\prime}\left(  r\right)  |+|Y\left(  r\right)
-Y^{\prime}\left(  r\right)  |>\delta\right\}  ,\medskip\\
A_{3}\left(  r\right)  :=\left\{  \omega:|X\left(  r\right)  |+|X^{\prime
}\left(  r\right)  |+|Y\left(  r\right)  |+|Y^{\prime}\left(  r\right)
|>M\right\}  .
\end{array}
\]
Using assumption $\left(  \mathrm{H}_{4}\right)  $ on $f$ and $h$ and Markov's
inequality we obtain the following estimations:

\begin{itemize}
\item[$\left(  a\right)  $]
\[%
\begin{array}
[c]{l}%
\displaystyle\int_{s}^{T}\mathbb{E}\left[  |f\left(  r\right)  -f^{\prime
}\left(  r\right)  |\mathbf{1}_{A_{1}\left(  r\right)  }\right]
dr+\mathbb{E}\left[  |h\left(  X\left(  T\right)  ,Y\left(  T\right)  \right)
-h(X^{\prime}\left(  T\right)  ,Y^{\prime}\left(  T\right)  )|\mathbf{1}%
_{A_{1}\left(  T\right)  }\right]  \medskip\\
\displaystyle\leq C\,\boldsymbol{\mu}_{f,h}\left(  \delta,M\right)  ;
\end{array}
\]

\item[$\left(  b\right)  $]
\[%
\begin{array}
[c]{l}%
\displaystyle\int_{s}^{T}\mathbb{E}\left[  |f\left(  r\right)  -f^{\prime
}\left(  r\right)  |\mathbf{1}_{A_{2}\left(  r\right)  }\right]
dr+\mathbb{E}\left[  |h\left(  X\left(  T\right)  ,Y\left(  T\right)  \right)
-h(X^{\prime}\left(  T\right)  ,Y^{\prime}\left(  T\right)  )|\mathbf{1}%
_{A_{2}\left(  T\right)  }\right]  \medskip\\
\displaystyle\leq\int_{s}^{T}\left[  \mathbb{E}|f\left(  r\right)  -f^{\prime
}\left(  r\right)  |^{2}\right]  ^{1/2}\left[  \mathbb{E}\mathbf{1}%
_{A_{2}\left(  r\right)  }\right]  ^{1/2}dr\medskip\\
\displaystyle\quad+\left[  \mathbb{E}|h\left(  X\left(  T\right)  ,Y\left(
T\right)  \right)  -h(X^{\prime}\left(  T\right)  ,Y^{\prime}\left(  T\right)
)|^{2}\right]  ^{1/2}\left[  \mathbb{E}\mathbf{1}_{A_{2}\left(  T\right)
}\right]  ^{1/2}\medskip\\
\displaystyle\leq\sqrt{2}\int_{s}^{T}\left[  \mathbb{E}|f\left(  r\right)
|^{2}+\mathbb{E}|f^{\prime}\left(  r\right)  |^{2}\right]  ^{1/2}\left[
\mathbb{P}A_{2}\left(  r\right)  \right]  ^{1/2}dr\medskip\\
\displaystyle\quad+\sqrt{2}\left[  \mathbb{E}|h\left(  X\left(  T\right)
,Y\left(  T\right)  \right)  |^{2}+\mathbb{E}|h(X^{\prime}\left(  T\right)
,Y^{\prime}\left(  T\right)  )|^{2}\right]  ^{1/2}\left[  \mathbb{P}%
A_{2}\left(  T\right)  \right]  ^{1/2}\medskip\\
\displaystyle\leq C\Big[1\mathbb{+E}\sup_{r\in\left[  s,T\right]
}\big(|X\left(  r\right)  |^{2p}+|Y\left(  r\right)  |^{2p}\mathbb{+}%
|X^{\prime}\left(  r\right)  |^{2p}\mathbb{+}|Y^{\prime}\left(  r\right)
|^{2p}\big)\Big]^{1/2}\cdot\medskip\\
\displaystyle\quad\bigg[\frac{\mathbb{E}\sup_{r\in\left[  s,T\right]
}\big(|X\left(  r\right)  -X^{\prime}\left(  r\right)  |^{2}+|Y\left(
r\right)  -Y^{\prime}\left(  r\right)  |^{2}\big)}{\delta^{2}}\bigg]^{1/2};
\end{array}
\]

\item[$\left(  c\right)  $]
\[%
\begin{array}
[c]{l}%
\displaystyle\int_{s}^{T}\mathbb{E}\left[  |f\left(  r\right)  -f^{\prime
}\left(  r\right)  |\mathbf{1}_{A_{3}\left(  r\right)  }\right]
dr+\mathbb{E}\left[  |h\left(  X\left(  T\right)  ,Y\left(  T\right)  \right)
-h(X^{\prime}\left(  T\right)  ,Y^{\prime}\left(  T\right)  )|\mathbf{1}%
_{A_{3}\left(  T\right)  }\right]  \medskip\\
\displaystyle\leq\int_{s}^{T}\left[  \mathbb{E}|f\left(  r\right)  -f^{\prime
}\left(  r\right)  |^{2}\right]  ^{1/2}\left[  \mathbb{E}\mathbf{1}%
_{A_{3}\left(  r\right)  }\right]  ^{1/2}dr\medskip\\
\displaystyle\quad+\left[  \mathbb{E}|h\left(  X\left(  T\right)  ,Y\left(
T\right)  \right)  -h(X^{\prime}\left(  T\right)  ,Y^{\prime}\left(  T\right)
)|^{2}\right]  ^{1/2}\left[  \mathbb{E}\mathbf{1}_{A_{3}\left(  T\right)
}\right]  ^{1/2}\medskip\\
\displaystyle\leq\sqrt{2}\int_{s}^{T}\left[  \mathbb{E}|f\left(  r\right)
|^{2}+\mathbb{E}|f^{\prime}\left(  r\right)  |^{2}\right]  ^{1/2}\left[
\mathbb{P}A_{3}\left(  r\right)  \right]  ^{1/2}dr\medskip\\
\displaystyle\quad+\sqrt{2}\left[  \mathbb{E}|h\left(  X\left(  T\right)
,Y\left(  T\right)  \right)  |^{2}+\mathbb{E}|h(X^{\prime}\left(  T\right)
,Y^{\prime}\left(  T\right)  )|^{2}\right]  ^{1/2}\left[  \mathbb{P}%
A_{3}\left(  T\right)  \right]  ^{1/2}\medskip\\
\displaystyle\leq C\Big[1\mathbb{+E}\sup_{r\in\left[  s,T\right]
}\big(|X\left(  r\right)  |^{2p}+|Y\left(  r\right)  |^{2p}\mathbb{+}%
|X^{\prime}\left(  r\right)  |^{2p}\mathbb{+}|Y^{\prime}\left(  r\right)
|^{2p}\big)\Big]^{1/2}\cdot\medskip\\
\displaystyle\quad\bigg[\frac{\mathbb{E}\sup_{r\in\left[  s,T\right]
}\big(|X\left(  r\right)  |^{2}+|Y\left(  r\right)  |^{2}+|X^{\prime}\left(
r\right)  |^{2}+|Y^{\prime}\left(  r\right)  |^{2}\big)}{M^{2}}\bigg]^{1/2};
\end{array}
\]

\item[$\left(  d\right)  $] and%
\[
\int_{s^{\prime}}^{s}\mathbb{E}|f^{\prime}\left(  r\right)  |dr\leq\bar
{\kappa}\left(  s-s^{\prime}\right)  \big(1+\mathbb{E}\sup_{r\in\lbrack
s^{\prime},s]}|X\left(  r\right)  |^{p}+\mathbb{E}\sup_{r\in\lbrack s^{\prime
},s]}|Y\left(  r\right)  |^{p}\Big).
\]

\end{itemize}

From the four inequalities hereabove, Theorem \ref{main result 2} and Remark
\ref{main result 2_remark}, we obtain%
\[%
\begin{array}
[c]{l}%
\displaystyle\left\vert J\left(  s,\xi;u\right)  -J(s^{\prime},\xi^{\prime
};u)\right\vert \leq C\,\boldsymbol{\mu}_{f,h}\left(  \delta,M\right)
+C\left(  s-s^{\prime}\right)  \big[1+\left\Vert \xi\right\Vert _{\mathcal{C}%
}^{p}\big]+C\big[1+\left\Vert \xi\right\Vert _{\mathcal{C}}^{p}+\left\Vert
\xi^{\prime}\right\Vert _{\mathcal{C}}^{p}\big]\cdot\medskip\\
\displaystyle\quad\quad\quad\quad\quad\quad\quad\quad\quad\quad\quad
\quad\bigg[\frac{\Gamma_{1}^{1/2}+\left\vert s-s^{\prime}\right\vert
^{1/2}\big(1+\left\Vert \xi\right\Vert _{\mathcal{C}}+||\xi^{\prime
}||_{\mathcal{C}}\big)}{\delta}+\frac{1+\left\Vert \xi\right\Vert
_{\mathcal{C}}+\left\Vert \xi^{\prime}\right\Vert _{\mathcal{C}}}%
{M}\bigg]\medskip\\
\displaystyle\leq C\,\boldsymbol{\mu}_{f,h}\left(  \delta,M\right)
+C\big[1+\left\Vert \xi\right\Vert _{\mathcal{C}}^{p}+\left\Vert \xi^{\prime
}\right\Vert _{\mathcal{C}}^{p}\big]\cdot\bigg[\frac{\Gamma_{1}^{1/2}%
+\left\vert s-s^{\prime}\right\vert ^{1/2}\big(1+\left\Vert \xi\right\Vert
_{\mathcal{C}}+||\xi^{\prime}||_{\mathcal{C}}\big)}{\delta}+\frac{1+\left\Vert
\xi\right\Vert _{\mathcal{C}}+\left\Vert \xi^{\prime}\right\Vert
_{\mathcal{C}}}{M}\bigg]
\end{array}
\]
(whenever $\delta$ and $\left\vert s^{\prime}-s\right\vert $ are from $\left(
0,1\right)  $).

The conclusion (\ref{V lipschitz}) follows now, since%
\[%
\begin{array}
[c]{l}%
|V\left(  s,\xi\right)  -V(s^{\prime},\xi^{\prime})|=|\sup_{u\in
\mathcal{U}\left[  s,T\right]  }(-J(s^{\prime},\xi^{\prime};u))-\sup
_{u\in\mathcal{U}\left[  s,T\right]  }\left(  -J\left(  s,\xi;u\right)
\right)  |\medskip\\
\leq|\sup_{u\in\mathcal{U}\left[  s,T\right]  }(J\left(  s,\xi;u\right)
-J(s^{\prime},\xi^{\prime};u))|\leq\sup_{u\in\mathcal{U}\left[  s,T\right]
}|(J\left(  s,\xi;u\right)  -J(s^{\prime},\xi^{\prime};u))|.
\end{array}
\]
The computations for inequality (\ref{V sublinear}) one uses the polynomial
growth of $f$ and $h$, inequalities (\ref{bound for sol}) and
(\ref{bound for sol 2}) and thus conclusion (\ref{V sublinear}) follows
easier.\hfill\bigskip
\end{proof}

In order to show that $V$ satisfies the DPP, we consider, for $\epsilon>0$,
the penalized equation:%
\begin{equation}
\left\{
\begin{array}
[c]{l}%
dX_{\epsilon}\left(  t\right)  +\nabla\varphi_{\epsilon}\left(  X_{\epsilon
}\left(  t\right)  \right)  dt=b\left(  t,X_{\epsilon}\left(  t\right)
,Y_{\epsilon}\left(  t\right)  ,Z_{\epsilon}\left(  t\right)  ,u\left(
t\right)  \right)  dt\medskip\\
\quad\quad\quad\quad\quad\quad\quad\quad\quad\quad\quad\quad+\sigma\left(
t,X_{\epsilon}\left(  t\right)  ,Y_{\epsilon}\left(  t\right)  ,Z_{\epsilon
}\left(  t\right)  ,u\left(  t\right)  \right)  dW(t),~t\in(s,T],\medskip\\
X_{\epsilon}\left(  t\right)  =\xi\left(  t-s\right)  ,\;t\in\left[
s-\delta,s\right]  ,
\end{array}
\right.  \label{SVI delayed approxim 2}%
\end{equation}
where%
\begin{equation}
Y_{\epsilon}(t):=\int_{-\delta}^{0}e^{\lambda r}X_{\epsilon}(t+r)dr,\quad
Z_{\epsilon}(t):=X_{\epsilon}(t-\delta) \label{def of Y,Z eps 2}%
\end{equation}
and we take the penalized value functions associated%
\begin{equation}%
\begin{array}
[c]{r}%
V_{\epsilon}\left(  s,\xi\right)  =\inf\nolimits_{u\in\mathcal{U}\left[
s,T\right]  }\mathbb{E}\big[\displaystyle\int_{s}^{T}f(t,X_{\epsilon}%
^{s,\xi,u}(t),Y_{\epsilon}^{s,\xi,u}(t),u(t))\,dt+h(X_{\epsilon}^{s,\xi
,u}(T),Y_{\epsilon}^{s,\xi,u}(T))\big],\medskip\\
\left(  s,\xi\right)  \in\lbrack0,T)\times\mathcal{C}\big(\left[
-\delta,0\right]  ;\overline{\mathrm{Dom}\left(  \varphi\right)  }\,\big).
\end{array}
\label{funct value 3}%
\end{equation}

\begin{remark}
Inequalities (\ref{V sublinear}) and (\ref{V lipschitz}) hold true for the
penalized value function $V_{\epsilon}$.
\end{remark}

The following result is a fairly straightforward generalization of Theorem 4.2
from \cite{la/02} to the case of $f$ and $h$ satisfying sublinear growth
(instead of lipschitzianity):

\begin{lemma}
\label{theorem 4.2}Let assumptions $\mathrm{(H}_{1}-\mathrm{H}_{4}\mathrm{)}$
be satisfied. If $X_{\epsilon}^{s,\xi,u}$ is the solution of
(\ref{SVI delayed approxim 2}), then, for every $\left(  s,\xi\right)
\in\lbrack0,T)\times\mathcal{C}\big(\left[  -\delta,0\right]  ;$
$\overline{\mathrm{Dom}\left(  \varphi\right)  }\,\big)$, it holds that%
\begin{equation}
V_{\epsilon}\left(  s,\xi\right)  =\inf_{u\in\mathcal{U}\left[  s,T\right]
}\mathbb{E}\big[\int_{s}^{\tau}f(r,X_{\epsilon}^{s,\xi,u}(r),Y_{\epsilon
}^{s,\xi,u}(r),u(r))dr+V_{\epsilon}(\tau,X_{\epsilon}^{s,\xi,u}\left(
\tau\right)  )\big], \label{DPP 2}%
\end{equation}
for every stopping time $\tau\in\left[  s,T\right]  .$
\end{lemma}

\begin{proof}
Since the stochastic controlled equation (\ref{SVI delayed approxim 2}) has
Lipschitz coefficients, we can use the proof of Theorem 4.2 with Lemma 4.1
replaced by inequality (\ref{V lipschitz}), written for $V_{\epsilon}$ (and
therefore a slight change of inequality (4.18) will appear).\hfill
\end{proof}

\begin{proposition}
Let assumptions $\mathrm{(H}_{1}-\mathrm{H}_{4}\mathrm{)}$ be satisfied. Then
there exists $C>0$ such that%
\begin{equation}%
\begin{array}
[c]{r}%
\displaystyle|V_{\epsilon}\left(  s,\xi\right)  -V(s,\xi)|\leq
C\,\boldsymbol{\mu}_{f,h}\left(  \delta,M\right)  +C\big[1+\left\Vert
\xi\right\Vert _{\mathcal{C}}^{p}\big](1+\varphi^{1/4}\left(  \xi\left(
0\right)  \right)  +\left\Vert \xi\right\Vert _{\mathcal{C}})\bigg[\frac
{\epsilon^{1/16}}{\delta}+\frac{1}{M}\bigg],\medskip\\
\forall\left(  s,\xi\right)  \in\left[  0,T\right]  \times\mathcal{C}%
\big(\left[  -\delta,0\right]  ;\overline{\mathrm{Dom}\left(  \varphi\right)
}\,\big)
\end{array}
\label{V approxim}%
\end{equation}
where%
\begin{equation}
\Gamma_{2}=1+\varphi^{2}\left(  \xi\left(  0\right)  \right)  +\left\Vert
\xi\right\Vert _{\mathcal{C}}^{4} \label{ineq bound 5_2}%
\end{equation}
(see definition (\ref{ineq bound 5})).
\end{proposition}

\begin{proof}
Passing to the limit in (\ref{Cauchy property of X}) we deduce that%
\begin{equation}
\mathbb{E}\sup_{r\in\left[  s,T\right]  }|X_{\epsilon}\left(  r\right)
-X\left(  r\right)  |^{2}\leq C\epsilon^{1/8}\Gamma_{2}^{1/4}
\label{approx of X}%
\end{equation}
and, using the same calculus type as in the proof of Proposition
\ref{properties of V} (see also (\ref{Cauchy property of Y})), we obtain%
\[%
\begin{array}
[c]{l}%
\displaystyle\left\vert J_{\epsilon}\left(  s,\xi;u\right)  -J(s,\xi
;u)\right\vert \leq C\,\boldsymbol{\mu}_{f,h}\left(  \delta,M\right)
+C\big[1+\left\Vert \xi\right\Vert _{\mathcal{C}}^{p}\big]\cdot\bigg[\frac
{\epsilon^{1/16}\Gamma_{2}^{1/8}}{\delta}+\frac{1+\left\Vert \xi\right\Vert
_{\mathcal{C}}}{M}\bigg]\medskip\\
\displaystyle\leq C\,\boldsymbol{\mu}_{f,h}\left(  \delta,M\right)
+C\big[1+\left\Vert \xi\right\Vert _{\mathcal{C}}^{p}\big]\cdot\bigg[\frac
{\epsilon^{1/16}(1+\varphi^{1/4}\left(  \xi\left(  0\right)  \right)
+\left\Vert \xi\right\Vert _{\mathcal{C}}^{1/2})}{\delta}+\frac{1+\left\Vert
\xi\right\Vert _{\mathcal{C}}}{M}\bigg]
\end{array}
\]
and, using Young's inequality, the conclusion follows.\hfill$\medskip$
\end{proof}

Using, mainly, inequalities (\ref{V lipschitz}) and (\ref{V approxim}) we can
be prove that

\begin{lemma}
\label{technical lemma}Function $V_{\epsilon}$ is uniformly convergent on
compacts to the value function $V$ on $\left[  0,T\right]  \times
\mathcal{C}\big(\left[  -\delta,0\right]  ;\overline{\mathrm{Dom}\left(
\varphi\right)  }\,\big).$
\end{lemma}

\begin{proof}
Let $\mathcal{K}$ be a compact nonempty subset of the convex domain
$\overline{\mathrm{Dom}\left(  \varphi\right)  }$ and $\eta>0$ be a arbitrary
small constant. We denote by $\mathcal{K}_{\eta}$ the $\eta$-interior of
$\mathcal{K}:$%
\[
\mathcal{K}_{\eta}=\big\{x\in\overline{\mathrm{Dom}\left(  \varphi\right)
}:\mathrm{dist}\left(  x,\mathrm{Bd}\left(  \mathrm{Dom}\left(  \varphi
\right)  \right)  \right)  >\eta\big\}.
\]
If we consider $\xi\in\mathcal{C}\left(  \left[  -\delta,0\right]
;\mathcal{K}\right)  $ then it is easy to see that there exists a function
$\xi^{\prime}\in\mathcal{C}\big(\left[  -\delta,0\right]  ;\mathrm{Int}\left(
\mathrm{Dom}\left(  \varphi\right)  \right)  \big)$ such that%
\[
\left\Vert \xi-\xi^{\prime}\right\Vert _{\mathcal{C}}\leq\eta.
\]
Now, from (\ref{V lipschitz}) and (\ref{V approxim}),%
\[%
\begin{array}
[c]{l}%
\displaystyle\left\vert V_{\epsilon}\left(  s,\xi\right)  -V\left(
s,\xi\right)  \right\vert \leq\left\vert V_{\epsilon}\left(  s,\xi\right)
-V_{\epsilon}\left(  s,\xi^{\prime}\right)  \right\vert +\left\vert
V_{\epsilon}\left(  s,\xi^{\prime}\right)  -V\left(  s,\xi^{\prime}\right)
\right\vert +\left\vert V\left(  s,\xi^{\prime}\right)  -V\left(
s,\xi\right)  \right\vert \medskip\\
\displaystyle\leq2C\,\boldsymbol{\mu}_{f,h}\left(  \delta,M\right)
+2C\big[1+\left\Vert \xi\right\Vert _{\mathcal{C}}^{p}+\left\Vert \xi^{\prime
}\right\Vert _{\mathcal{C}}^{p}\big]\bigg[\frac{\Gamma_{1}^{1/2}+\left\vert
s-s^{\prime}\right\vert ^{1/2}\big(1+\left\Vert \xi\right\Vert _{\mathcal{C}%
}+||\xi^{\prime}||_{\mathcal{C}}\big)}{\delta}+\frac{1+\left\Vert
\xi\right\Vert _{\mathcal{C}}+\left\Vert \xi^{\prime}\right\Vert
_{\mathcal{C}}}{M}\bigg]\medskip\\
\displaystyle\quad+C\,\boldsymbol{\mu}_{f,h}\left(  \bar{\delta},\bar
{M}\right)  +C\big[1+\left\Vert \xi^{\prime}\right\Vert _{\mathcal{C}}%
^{p}\big](1+\varphi^{1/4}\left(  \xi^{\prime}\left(  0\right)  \right)
+\left\Vert \xi^{\prime}\right\Vert _{\mathcal{C}})\Big[\frac{\epsilon^{1/16}%
}{\delta}+\frac{1}{M}\Big].
\end{array}
\]
Therefore we can chose $\delta>0$ and $M,\bar{M}>0$ such that%
\[
\limsup_{\epsilon\rightarrow0}\sup_{\left[  0,T\right]  \times\mathcal{C}%
\left(  \left[  -\delta,0\right]  ;\mathcal{K}\right)  }\left\vert
V_{\epsilon}\left(  s,\xi\right)  -V\left(  s,\xi\right)  \right\vert \leq
C\eta,~\text{for all }\eta>0
\]
and the conclusion follows.\hfill
\end{proof}

The main result of this section is the following:

\begin{proposition}
Under the assumptions $\mathrm{(H}_{1}-\mathrm{H}_{4}\mathrm{)}$ the value
function $V$ satisfies the DPP.
\end{proposition}

\begin{proof}
From Lemma \ref{theorem 4.2} we see that $V_{\epsilon}$ satisfies the DPP
(\ref{DPP 2}). Let now $\left(  s,\xi\right)  \in\lbrack0,T)\times
\mathcal{C}\big(\left[  -\delta,0\right]  ;\mathrm{Dom}(\varphi)\big)$ (hence
$\varphi\left(  \xi\left(  r\right)  \right)  <+\infty$) be arbitrary but
fixed. We have, for every $\epsilon>0$, $u\in\mathcal{U}\left[  s,T\right]  $
and any stopping time $\tau\in\left[  s,T\right]  $ and $M>0$,%
\begin{equation}%
\begin{array}
[c]{l}%
\mathbb{E}\left\vert V_{\epsilon}(\tau,X_{\epsilon}^{s,\xi,u}(\tau
))-V(\tau,X^{s,\xi,u}(\tau))\right\vert \medskip\\
\leq\mathbb{E}\left\vert V_{\epsilon}(\tau,X_{\epsilon}^{s,\xi,u}%
(\tau))-V_{\epsilon}(\tau,X^{s,\xi,u}(\tau))\right\vert +\mathbb{E}\left[
\left\vert V_{\epsilon}(\tau,X^{s,\xi,u}(\tau))-V(\tau,X^{s,\xi,u}%
(\tau))\right\vert \left(  \mathbf{1}_{A_{1}}+\mathbf{1}_{A_{2}}\right)
\right]  \medskip\\
\leq\sup\limits_{\left(  t,y\right)  \in\mathrm{B}}\mathbb{E}\left\vert
V_{\epsilon}(t,y)-V(t,y)\right\vert +\mathbb{E}\left\vert V_{\epsilon}%
(\tau,X_{\epsilon}^{s,\xi,u}(\tau))-V_{\epsilon}(\tau,X^{s,\xi,u}%
(\tau))\right\vert \medskip\\
\quad+\mathbb{E}\left[  \left\vert V_{\epsilon}(\tau,X^{s,\xi,u}(\tau
))-V(\tau,X^{s,\xi,u}(\tau))\right\vert \mathbf{1}_{A_{2}}\right]  ,
\end{array}
\label{V approxim 2}%
\end{equation}
where%
\[
A_{1}:=\left\{  \omega:|X^{s,\xi,u}\left(  \tau\right)  |\leq M\right\}
~,~~A_{2}:=\left\{  \omega:|X^{s,\xi,u}\left(  \tau\right)  |>M\right\}
\]
and%
\[
\mathrm{B}:=\left[  0,T\right]  \times\mathcal{C}\big(\left[  -\delta
,0\right]  ;\overline{\mathcal{B}\left(  0,M\right)  }\cap\overline
{\mathrm{Dom}\left(  \varphi\right)  }\,\big).
\]
In order to obtain an estimate for the term $\mathbb{E}\left\vert V_{\epsilon
}(\tau,X_{\epsilon}^{s,\xi,u}(\tau))-V_{\epsilon}(\tau,X^{s,\xi,u}%
(\tau))\right\vert $ we do similar computations as in proof of Proposition
\ref{properties of V}%
\begin{equation}%
\begin{array}
[c]{l}%
|(J_{\epsilon}(\tau,X_{\epsilon}^{s,\xi,u}(\tau);u)-J_{\epsilon}(\tau
,X^{s,\xi,u}(\tau);u))|\medskip\\
\leq\displaystyle\int_{\tau}^{T}\mathbb{E}|f_{\epsilon}^{1}\left(  r\right)
-f_{\epsilon}^{2}\left(  r\right)  |dr+\mathbb{E}|h(X_{\epsilon}^{1}\left(
T\right)  ,Y_{\epsilon}^{1}\left(  T\right)  )-h\left(  X_{\epsilon}%
^{2}\left(  T\right)  ,Y_{\epsilon}^{2}\left(  T\right)  \right)  |\medskip\\
\displaystyle\leq C\,\boldsymbol{\mu}_{f,h}\left(  \delta,M\right)
+C\Big[1\mathbb{+E}\sup_{r\in\left[  \tau,T\right]  }\big(|X_{\epsilon}%
^{1}\left(  r\right)  |^{2p}+|Y_{\epsilon}^{1}\left(  r\right)  |^{2p}%
\mathbb{+}|X_{\epsilon}^{2}\left(  r\right)  |^{2p}\mathbb{+}|Y_{\epsilon}%
^{2}\left(  r\right)  |^{2p}\big)\Big]^{1/2}\cdot\medskip\\
\displaystyle\quad\bigg[\frac{\left(  \mathbb{E}\sup_{r\in\left[
\tau,T\right]  }(|X_{\epsilon}^{1}\left(  r\right)  -X_{\epsilon}^{2}\left(
r\right)  |^{2}+|Y_{\epsilon}^{1}\left(  r\right)  -Y_{\epsilon}^{2}\left(
r\right)  |^{2})\right)  ^{1/2}}{\delta}\medskip\\
\displaystyle\quad\quad\quad\quad\quad\quad\quad\quad\quad\quad+\frac{\left(
\mathbb{E}\sup_{r\in\left[  \tau,T\right]  }(|X_{\epsilon}^{1}\left(
r\right)  |^{2}+|Y_{\epsilon}^{1}\left(  r\right)  |^{2}+|X_{\epsilon}%
^{2}\left(  r\right)  |^{2}+|Y_{\epsilon}^{2}\left(  r\right)  |^{2})\right)
^{1/2}}{M}\bigg],
\end{array}
\label{differ of 2 J_eps}%
\end{equation}
where%
\begin{align*}
X_{\epsilon}^{1}\left(  r\right)   &  :=X_{\epsilon}^{\tau,X_{\epsilon}%
^{s,\xi,u}\left(  \tau\right)  ,u}\left(  r\right)  \,,\quad X_{\epsilon}%
^{2}\left(  r\right)  :=X_{\epsilon}^{\tau,X^{s,\xi,u}\left(  \tau\right)
,u}\left(  r\right)  ,\\
Y_{\epsilon}^{1}\left(  r\right)   &  :=Y_{\epsilon}^{\tau,X_{\epsilon}%
^{s,\xi,u}\left(  \tau\right)  ,u}\left(  r\right)  \,,\quad Y_{\epsilon}%
^{2}\left(  r\right)  :=Y_{\epsilon}^{\tau,X^{s,\xi,u}\left(  \tau\right)
,u}\left(  r\right)
\end{align*}
and%
\[
f_{\epsilon}^{1}\left(  r\right)  =f(r,X_{\epsilon}^{1}\left(  r\right)
,Y_{\epsilon}^{1}\left(  r\right)  ,u\left(  r\right)  )\,,\quad f_{\epsilon
}^{2}\left(  r\right)  =f(r,X_{\epsilon}^{2}\left(  r\right)  ,Y_{\epsilon
}^{2}\left(  r\right)  ,u\left(  r\right)  ).
\]
Since%
\[%
\begin{array}
[c]{l}%
\mathbb{E}\sup_{r\in\left[  \tau,T\right]  }(|X_{\epsilon}^{1}\left(
r\right)  |^{2p}+|Y_{\epsilon}^{1}\left(  r\right)  |^{2p}+|X_{\epsilon}%
^{2}\left(  r\right)  |^{2p}+|Y_{\epsilon}^{2}\left(  r\right)  |^{2p}%
)\medskip\\
\leq C\big(1+\mathbb{E}\sup_{r\in\left[  \tau-\delta,\tau\right]
}|X_{\epsilon}^{s,\xi,u}\left(  r\right)  |^{2p}+\mathbb{E}\sup_{r\in\left[
\tau-\delta,\tau\right]  }|X^{s,\xi,u}\left(  r\right)  |^{2p}\big)\leq
C\big(1+\left\Vert \xi\right\Vert _{\mathcal{C}}^{2p}\big)
\end{array}
\]
and%
\[%
\begin{array}
[c]{l}%
\mathbb{E}\sup_{r\in\left[  \tau,T\right]  }(|X_{\epsilon}^{1}\left(
r\right)  -X_{\epsilon}^{2}\left(  r\right)  |^{2}+|Y_{\epsilon}^{1}\left(
r\right)  -Y_{\epsilon}^{2}\left(  r\right)  |^{2})\medskip\\
\leq C\mathbb{E}\sup_{r\in\left[  \tau-\delta,\tau\right]  }|X_{\epsilon
}^{s,\xi,u}\left(  r\right)  -X^{s,\xi,u}\left(  r\right)  |^{2}\leq
C\epsilon^{1/8}\Gamma_{2}^{1/4},
\end{array}
\]
we obtain%
\[%
\begin{array}
[c]{l}%
|(J_{\epsilon}(\tau,X_{\epsilon}^{s,\xi,u}(\tau);u)-J_{\epsilon}(\tau
,X^{s,\xi,u}(\tau);u))|\medskip\\
\displaystyle\leq C\,\boldsymbol{\mu}_{f,h}\left(  \delta,M\right)
+C\big[1+\left\Vert \xi\right\Vert _{\mathcal{C}}^{p}\big](1+\varphi
^{1/4}\left(  \xi\left(  0\right)  \right)  +\left\Vert \xi\right\Vert
_{\mathcal{C}})\bigg[\frac{\epsilon^{1/16}}{\delta}+\frac{1}{M}\bigg].
\end{array}
\]
Hence%
\[%
\begin{array}
[c]{l}%
\mathbb{E}\left\vert V_{\epsilon}(\tau,X_{\epsilon}^{s,\xi,u}(\tau
))-V_{\epsilon}(\tau,X^{s,\xi,u}(\tau))\right\vert \leq\sup_{u\in
\mathcal{U}\left[  s,T\right]  }\mathbb{E}|(J_{\epsilon}(\tau,X_{\epsilon
}^{s,\xi,u}(\tau);u)-J_{\epsilon}(\tau,X^{s,\xi,u}(\tau);u))|\medskip\\
\displaystyle\leq C\,\boldsymbol{\mu}_{f,h}\left(  \delta,M\right)
+C\big[1+\left\Vert \xi\right\Vert _{\mathcal{C}}^{p}\big](1+\varphi
^{1/4}\left(  \xi\left(  0\right)  \right)  +\left\Vert \xi\right\Vert
_{\mathcal{C}})\Big[\frac{\epsilon^{1/16}}{\delta}+\frac{1}{M}\Big].
\end{array}
\]
For estimation of the term $\mathbb{E}\left[  \left\vert V_{\epsilon}%
(\tau,X^{s,\xi,u}(\tau))-V(\tau,X^{s,\xi,u}(\tau))\right\vert \mathbf{1}%
_{A_{2}}\right]  $ we use Markov's inequality and we see that%
\[%
\begin{array}
[c]{l}%
\mathbb{E}\left[  \left\vert V_{\epsilon}(\tau,X^{s,\xi,u}(\tau))-V(\tau
,X^{s,\xi,u}(\tau))\right\vert \mathbf{1}_{A_{2}}\right]  \leq\left[
\mathbb{E}\left\vert V_{\epsilon}(\tau,X^{s,\xi,u}(\tau))-V(\tau,X^{s,\xi
,u}(\tau))\right\vert \right]  ^{1/2}\left[  \mathbb{E}\left(  \mathbf{1}%
_{A_{2}}\right)  \right]  ^{1/2}\medskip\\
\leq\sqrt{2}\left[  \mathbb{E}\left\vert V_{\epsilon}(\tau,X^{s,\xi,u}%
(\tau))\right\vert ^{2}+\mathbb{E}\left\vert V(\tau,X^{s,\xi,u}(\tau
))\right\vert ^{2}\right]  ^{1/2}\displaystyle\frac{\left[  \mathbb{E}%
\left\vert X^{s,\xi,u}(\tau)\right\vert ^{2}\right]  ^{1/2}}{M}\,.
\end{array}
\]
Now, using (\ref{V sublinear}) and (\ref{bound for sol}) (still true for the
approximating sequence $X_{\epsilon}^{\tau,X^{s,\xi,u}\left(  \tau\right)
,u}$ ),%
\[
\mathbb{E}\left\vert V_{\epsilon}(\tau,X^{s,\xi,u}(\tau))\right\vert
^{2}+\mathbb{E}\left\vert V(\tau,X^{s,\xi,u}(\tau))\right\vert ^{2}\leq
C\big(1+\mathbb{E}\sup_{r\in\left[  \tau-\delta,\tau\right]  }|X^{s,\xi
,u}\left(  r\right)  |^{2p}\big)\leq C\big(1+\left\Vert \xi\right\Vert
_{\mathcal{C}}^{2p}\big).
\]
Therefore inequality (\ref{V approxim 2}) becomes%
\begin{equation}%
\begin{array}
[c]{l}%
\mathbb{E}\left\vert V_{\epsilon}(\tau,X_{\epsilon}^{s,\xi,u}(\tau
))-V(\tau,X^{s,\xi,u}(\tau))\right\vert \medskip\\
\leq\sup\limits_{\left(  t,y\right)  \in\mathrm{B}}\mathbb{E}\left\vert
V_{\epsilon}(t,y)-V(t,y)\right\vert +C\,\boldsymbol{\mu}_{f,h}\left(
\delta,M\right)  +C\big(1+\left\Vert \xi\right\Vert _{\mathcal{C}}%
^{p+1}\big)\displaystyle\frac{1}{M}\medskip\\
\quad\displaystyle+C\big[1+\left\Vert \xi\right\Vert _{\mathcal{C}}%
^{p}\big](1+\varphi^{1/4}\left(  \xi\left(  0\right)  \right)  +\left\Vert
\xi\right\Vert _{\mathcal{C}})\Big[\frac{\epsilon^{1/16}}{\delta}+\frac{1}%
{M}\Big]\,.
\end{array}
\label{V approxim 3}%
\end{equation}
We pass to the proof of DPP (\ref{DPP}). Let $\bar{V}\left(  s,\xi\right)  $
denote the right term from (\ref{DPP}).

Using (\ref{DPP 2}), (\ref{V approxim 3}), inequality%
\[
\mathbb{E}\big[V_{\epsilon}(\tau,X_{\epsilon}^{s,\xi,u}\left(  \tau\right)
)-V(\tau,X_{\epsilon}^{s,\xi,u}\left(  \tau\right)  )\big]\leq\mathbb{E}%
\big[\big|V_{\epsilon}(\tau,X_{\epsilon}^{s,\xi,u}\left(  \tau\right)
)-V(\tau,X_{\epsilon}^{s,\xi,u}\left(  \tau\right)  )\big|\big]
\]
and%
\[%
\begin{array}
[c]{l}%
\displaystyle\mathbb{E}\big[\int_{s}^{\tau}f(r,X_{\epsilon}^{s,\xi
,u}(r),Y_{\epsilon}^{s,\xi,u}(r),u(r))dr-\int_{s}^{\tau}f(r,X^{s,\xi
,u}(r),Y^{s,\xi,u}(r),u(r))dr\big]\medskip\\
\displaystyle\leq\mathbb{E}\big[\int_{s}^{\tau}\left\vert f(r,X_{\epsilon
}^{s,\xi,u}(r),Y_{\epsilon}^{s,\xi,u}(r),u(r))-f(r,X^{s,\xi,u}(r),Y^{s,\xi
,u}(r),u(r))\right\vert dr\big]\medskip\\
\displaystyle\leq C\,\boldsymbol{\mu}_{f,h}\left(  \delta,M\right)
+C\big[1+\left\Vert \xi\right\Vert _{\mathcal{C}}^{p}\big](1+\varphi
^{1/4}\left(  \xi\left(  0\right)  \right)  +\left\Vert \xi\right\Vert
_{\mathcal{C}})\Big[\frac{\epsilon^{1/16}}{\delta}+\frac{1}{M}\Big]
\end{array}
\]
we deduce%
\begin{equation}%
\begin{array}
[c]{l}%
V\left(  s,\xi\right)  \leq V_{\epsilon}\left(  s,\xi\right)  +\left\vert
V_{\epsilon}\left(  s,\xi\right)  -V\left(  s,\xi\right)  \right\vert
\medskip\\
\displaystyle\leq\mathbb{E}\big[\int_{s}^{\tau}f(r,X_{\epsilon}^{s,\xi
,u}(r),Y_{\epsilon}^{s,\xi,u}(r),u(r))dr+V_{\epsilon}(\tau,X_{\epsilon}%
^{s,\xi,u}\left(  \tau\right)  )\big]+\left\vert V_{\epsilon}\left(
s,\xi\right)  -V\left(  s,\xi\right)  \right\vert \medskip\\
\displaystyle\leq\mathbb{E}\big[\int_{s}^{\tau}f(r,X^{s,\xi,u}(r),Y^{s,\xi
,u}(r),u(r))dr+V(\tau,X_{\epsilon}^{s,\xi,u}\left(  \tau\right)
)\big]+\left\vert V_{\epsilon}\left(  s,\xi\right)  -V\left(  s,\xi\right)
\right\vert \medskip\\
\displaystyle\quad+\sup\limits_{\left(  t,y\right)  \in\mathrm{B}}%
\mathbb{E}\left\vert V_{\epsilon}(t,y)-V(t,y)\right\vert +C\,\boldsymbol{\mu
}_{f,h}\left(  \delta,M\right)  +C\big(1+\left\Vert \xi\right\Vert
_{\mathcal{C}}^{p+1}\big)\displaystyle\frac{1}{M}\\
\displaystyle\quad+C\big[1+\left\Vert \xi\right\Vert _{\mathcal{C}}%
^{p}\big](1+\varphi^{1/4}\left(  \xi\left(  0\right)  \right)  +\left\Vert
\xi\right\Vert _{\mathcal{C}})\Big[\frac{\epsilon^{1/16}}{\delta}+\frac{1}%
{M}\Big]\,.
\end{array}
\label{DPP 3}%
\end{equation}
Passing to the limit for $\epsilon\rightarrow0$, $\delta\rightarrow0$ and
$M\rightarrow+\infty$ we obtain%
\begin{equation}
V\left(  s,\xi\right)  \leq\bar{V}\left(  s,\xi\right)  .\label{DPP 4}%
\end{equation}
Conversely, let $\eta>0$. Since $V_{\epsilon}$ satisfies the DPP, there exists
$u_{\epsilon}\in\mathcal{U}\left[  s,T\right]  $ such that%
\[
V_{\epsilon}\left(  s,\xi\right)  +\frac{\eta}{2}\geq\mathbb{E}\big[\int
_{s}^{\tau}f(r,X_{\epsilon}^{s,\xi,u_{\epsilon}}(r),Y_{\epsilon}%
^{s,\xi,u_{\epsilon}}(r),u_{\epsilon}(r))dr+V_{\epsilon}(\tau,X_{\epsilon
}^{s,\xi,u_{\epsilon}}\left(  \tau\right)  )\big]
\]
and, as in the proof of (\ref{DPP 3}), we deduce that%
\[%
\begin{array}
[c]{l}%
V\left(  s,\xi\right)  +\eta\geq V_{\epsilon}\left(  s,\xi\right)
+\eta-\left\vert V_{\epsilon}\left(  s,\xi\right)  -V\left(  s,\xi\right)
\right\vert \medskip\\
\displaystyle\geq\mathbb{E}\big[\int_{s}^{\tau}f(r,X_{\epsilon}^{s,\xi
,u_{\epsilon}}(r),Y_{\epsilon}^{s,\xi,u_{\epsilon}}(r),u_{\epsilon
}(r))dr+V_{\epsilon}(\tau,X_{\epsilon}^{s,\xi,u_{\epsilon}}\left(
\tau\right)  )\big]+\frac{\eta}{2}-\left\vert V_{\epsilon}\left(
s,\xi\right)  -V\left(  s,\xi\right)  \right\vert \medskip\\
\displaystyle\geq\mathbb{E}\big[\int_{s}^{\tau}f(r,X^{s,\xi,u_{\epsilon}%
}(r),Y^{s,\xi,u_{\epsilon}}(r),u(r))dr+V(\tau,X_{\epsilon}^{s,\xi,u_{\epsilon
}}\left(  \tau\right)  )\big]+\frac{\eta}{2}-\left\vert V_{\epsilon}\left(
s,\xi\right)  -V\left(  s,\xi\right)  \right\vert \medskip\\
\displaystyle\quad-\sup\limits_{\left(  t,y\right)  \in\mathrm{B}}%
\mathbb{E}\left\vert V_{\epsilon}(t,y)-V(t,y)\right\vert -C\,\boldsymbol{\mu
}_{f,h}\left(  \delta,M\right)  -C\big(1+\left\Vert \xi\right\Vert
_{\mathcal{C}}^{p+1}\big)\displaystyle\frac{1}{M}\\
\displaystyle\quad-C\big[1+\left\Vert \xi\right\Vert _{\mathcal{C}}%
^{p}\big](1+\varphi^{1/4}\left(  \xi\left(  0\right)  \right)  +\left\Vert
\xi\right\Vert _{\mathcal{C}})\Big[\frac{\epsilon^{1/16}}{\delta}+\frac{1}%
{M}\Big]\,.
\end{array}
\]
Therefore,%
\begin{equation}
V\left(  s,\xi\right)  \geq\bar{V}\left(  s,\xi\right)  .\label{DPP 5}%
\end{equation}
The proof is completed by showing that inequalities (\ref{DPP 4}) and
(\ref{DPP 5}) can be extended for any $\left(  s,\xi\right)  \in
\lbrack0,T)\times\mathcal{C}\big(\left[  -\delta,0\right]  ;\overline
{\mathrm{Dom}\left(  \varphi\right)  }\,\big).$\hfill
\end{proof}

\subsection{Hamilton-Jacobi-Bellman Equation. Viscosity solution}

Since $V$ is defined on $[0,T]\times\mathcal{C}\big(\left[  -\delta,0\right]
;\overline{\mathrm{Dom}\left(  \varphi\right)  }\,\big)$, the associated
Hamilton-Jacobi-Bellman equation we will be an infinite dimensional PDE. In
general the value function $V\left(  s,\xi\right)  $ depend on the initial
path in a complicated way. In order to simplify the problem, our conjecture
will be that the value function $V$ depend on $\xi$ only through $\left(
x,y\right)  $ where%
\[
x=x\left(  \xi\right)  :=\xi\left(  0\right)  \quad\text{and}\quad y=y\left(
\xi\right)  :=\int_{-\delta}^{0}e^{\lambda r}\xi\left(  r\right)  dr.
\]
Hence the problem can be reduced to a finite dimensional optimal control
problem by working with a new value function $\tilde{V}$ given by%
\[
\tilde{V}:\left[  0,T\right]  \times\mathbb{R}^{2d}\rightarrow\mathbb{R}%
\text{,}\quad\tilde{V}\left(  s,x,y\right)  :=V\left(  s,\xi\right)  .
\]
Our aim is to prove that the value function $\tilde{V}$ is a viscosity
solution of the following Hamilton-Jacobi-Bellman type PDE%
\begin{equation}
\left\{
\begin{array}
[c]{r}%
\displaystyle-\frac{\partial\tilde{V}}{\partial s}(s,x,y)+\sup_{u\in
\mathrm{U}}\mathcal{H}\big(s,x,y,z,u,-D_{x}\tilde{V}\left(  s,x,y\right)
,-D_{xx}^{2}\tilde{V}\left(  s,x,y\right)  \big)\medskip\\
-\langle x-e^{-\lambda\delta}z-\lambda y,D_{y}\tilde{V}\left(  s,x,y\right)
\rangle\in\langle-D_{x}\tilde{V}\left(  s,x,y\right)  ,\partial\varphi\left(
x\right)  \rangle,\medskip\\
\text{for }\left(  s,x,y,z\right)  \in\left(  0,T\right)  \times
\overline{\mathrm{Dom}\left(  \varphi\right)  }\times\mathbb{R}^{2d}%
,\medskip\\
\multicolumn{1}{l}{\tilde{V}\left(  T,x,y\right)  =h\left(  x,y\right)  \text{
for }\left(  x,y\right)  \in\overline{\mathrm{Dom}\left(  \varphi\right)
}\times\mathbb{R}^{d},}%
\end{array}
\right.  \label{HJB}%
\end{equation}
where $\mathcal{H}:\left[  0,T\right]  \times\mathbb{R}^{3d}\times
\mathrm{U}\times\mathbb{R}^{d}\times\mathbb{R}^{d\times d}\rightarrow
\mathbb{R}$ is defined by%
\[
\mathcal{H}\left(  s,x,y,z,u,q,X\right)  :=\left\langle b\left(
s,x,y,z,u\right)  ,q\right\rangle +\frac{1}{2}\mathrm{Tr}\left(  \sigma
\sigma^{\ast}\right)  \left(  s,x,y,z,u\right)  X-f\left(  s,x,y,u\right)  .
\]
Let us define, for $x\in\overline{\mathrm{Dom}\left(  \varphi\right)  }$ and
$z\in\mathbb{R}^{d}$,%
\begin{equation}
\partial\varphi_{\ast}(x;z)=\liminf_{\substack{(x^{\prime},z^{\prime
})\rightarrow(x,z)\\x^{\ast}\in\partial\varphi(x^{\prime})}}\left\langle
x^{\ast},z^{\prime}\right\rangle \quad\text{and}\quad\partial\varphi^{\ast
}(x;z)=\limsup_{_{_{\substack{(x^{\prime},z^{\prime})\rightarrow
(x,z)\\x^{\ast}\in\partial\varphi(x^{\prime})}}}}\left\langle x^{\ast
},z^{\prime}\right\rangle \,. \label{def for visc}%
\end{equation}

\begin{remark}
Obviously, $\partial\varphi^{\ast}(x;z)=-\partial\varphi_{\ast}(x;-z).$
\end{remark}

The following technical result is due to \cite{za/12}:

\begin{lemma}
$\left(  i\right)  $ For any $x\in\mathrm{Int}\left(  \mathrm{Dom}\left(
\varphi\right)  \right)  $ and $z\in\mathbb{R}^{d}$%
\begin{equation}
\partial\varphi_{\ast}(x;z)=\inf_{x^{\ast}\in\partial\varphi(x)}\left\langle
x^{\ast},z\right\rangle \label{form of phi}%
\end{equation}
$\left(  ii\right)  $ For any $x\in\mathrm{Bd}\left(  \mathrm{Dom}\left(
\varphi\right)  \right)  $ and $z\in\mathbb{R}^{d}$ such that%
\[
\inf_{\mathrm{n}\in\mathcal{N}(x)}\left\langle \mathrm{n},z\right\rangle >0
\]
equality (\ref{form of phi}) still holds$\smallskip$

\noindent(here $\mathcal{N}(x)$ denotes the exterior normal cone of versors,
in a point $x$ which belongs to the boundary of the domain).
\end{lemma}

It is easy to see that in the particular case of $\varphi$ being the indicator
function of a closed convex set $\mathcal{K}$ (i.e. $\varphi\left(  x\right)
=0$, if $x\in\mathcal{K}$ and $\varphi\left(  x\right)  =+\infty$ if
$x\notin\mathcal{K}$), we obtain the form:%
\[
\partial\varphi_{\ast}(x;z)=\left\{
\begin{array}
[c]{rl}%
0, & \text{if }x\in\mathrm{Int}\left(  \mathrm{Dom}\left(  \varphi\right)
\right)  \text{ or if }x\in\mathrm{Bd}\left(  \mathrm{Dom}\left(
\varphi\right)  \right)  \text{ with }\inf_{\mathrm{n}\in\mathcal{N}%
(x)}\left\langle \mathrm{n},z\right\rangle >0,\medskip\\
-\infty, & \text{if }\inf_{\mathrm{n}\in\mathcal{N}(x)}\left\langle
\mathrm{n},z\right\rangle \leq0.
\end{array}
\right.
\]

We define the viscosity solution for HJB equation (\ref{HJB}):

\begin{definition}
Let $v:(0,T]\times\overline{\mathrm{Dom}\left(  \varphi\right)  }%
\times\mathbb{R}^{d}\rightarrow\mathbb{R}$ be a continuous function which
satisfies $v(T,x,y)=h\left(  x\right)  ,\;\forall~\left(  x,y\right)
\in\overline{\mathrm{Dom}\left(  \varphi\right)  }\times\mathbb{R}^{d}%
.$\medskip

\noindent$\left(  a\right)  $ We say that $v$ is a viscosity subsolution of
(\ref{HJB}) if in any point $\left(  s,x,y\right)  \in(0,T]\times
\overline{\mathrm{Dom}\left(  \varphi\right)  }\times\mathbb{R}^{d}$ which is
a maximum point for $v-\Psi$, where $\Psi\in C^{1,2,1}(\left(  0,T\right)
\times\overline{\mathrm{Dom}\left(  \varphi\right)  }\times\mathbb{R}%
^{d};\mathbb{R)}$, the following inequality is satisfied:%
\[%
\begin{array}
[c]{r}%
-\dfrac{\partial\Psi}{\partial t}\left(  s,x,y\right)  +\sup_{u\in\mathrm{U}%
}\mathcal{H}\left(  s,x,y,z,u,-D_{x}\Psi\left(  s,x,y\right)  ,-D_{xx}^{2}%
\Psi\left(  s,x,y\right)  \right)  \medskip\\
-\left\langle x-e^{-\lambda\delta}z-\lambda y,D_{y}\Psi\left(  s,x,y\right)
\right\rangle \leq\partial\varphi^{\ast}(x;-D_{x}\Psi(t,x,y))\,.
\end{array}
\]
\noindent$\left(  b\right)  $\textit{ }We say that $v$ is a viscosity
supersolution of (\ref{HJB}) if in any point $\left(  s,x,y\right)
\in(0,T]\times\overline{\mathrm{Dom}\left(  \varphi\right)  }\times
\mathbb{R}^{d}$ which is a minimum point for $v-\Psi$, where $\Psi\in
C^{1,2,1}(\left(  0,T\right)  \times\overline{\mathrm{Dom}\left(
\varphi\right)  }\times\mathbb{R}^{d};\mathbb{R)}$, the following inequality
is satisfied:%
\[%
\begin{array}
[c]{r}%
-\dfrac{\partial\Psi}{\partial t}\left(  s,x,y\right)  +\sup_{u\in\mathrm{U}%
}\mathcal{H}\left(  s,x,y,z,u,-D_{x}\Psi\left(  s,x,y\right)  ,-D_{xx}%
\Psi\left(  s,x,y\right)  \right)  \medskip\\
-\left\langle x-e^{-\lambda\delta}z-\lambda y,D_{y}\Psi\left(  s,x,y\right)
\right\rangle \geq\partial\varphi_{\ast}(x;-D_{x}\Psi(t,x,y))\,.
\end{array}
\]
\noindent$(c)$ We say that $v$ is a viscosity solution of (\ref{HJB}) if it is
both a viscosity sub- and super-solution.
\end{definition}

\begin{theorem}
\label{existence of visc}Under the assumptions $\mathrm{(H}_{1}-\mathrm{H}%
_{4}\mathrm{)}$ the value function $\tilde{V}$ is a viscosity solution of
(\ref{HJB}).
\end{theorem}

\begin{proof}
Let $\Psi\in C^{1,2,1}(\left(  0,T\right)  \times\overline{\mathrm{Dom}\left(
\varphi\right)  }\times\mathbb{R}^{d};\mathbb{R)}$ such that $V-\Psi$ has a
local maximum in $\left(  s,x,y\right)  \in(0,T]\times\overline{\mathrm{Dom}%
\left(  \varphi\right)  }\times\mathbb{R}^{d}.$ If $\epsilon>0$ it is easy to
deduce that there exists $\left(  s_{\epsilon},x_{\epsilon},y_{\epsilon
}\right)  \in(0,T]\times\overline{\mathrm{Dom}\left(  \varphi\right)  }%
\times\mathbb{R}^{d}$ such that $V_{\epsilon}-\Psi$ has a local maximum in
$\left(  s_{\epsilon},x_{\epsilon},y_{\epsilon}\right)  $ and $\left(
s_{\epsilon},x_{\epsilon},y_{\epsilon}\right)  \rightarrow\left(
s,x,y\right)  $, as $\epsilon\rightarrow0$.

Applying Theorem 4.1 from \cite{la-ri/03} (with $b\left(  s,x,y,z,u\right)  $
replaced by $b\left(  s,x,y,z,u\right)  -\nabla\varphi_{\epsilon}\left(
x\right)  $) we deduce that $V_{\epsilon}$ is a viscosity subsolution of HJB
equation%
\[
\left\{
\begin{array}
[c]{r}%
\displaystyle-\frac{\partial\tilde{V}}{\partial s}(s,x,y)+\sup_{u\in
\mathrm{U}}\mathcal{H}\big(s,x,y,z,u,-D_{x}\tilde{V}\left(  s,x,y\right)
,-D_{xx}^{2}\tilde{V}\left(  s,x,y\right)  \big)\medskip\\
-\langle x-e^{-\lambda\delta}z-\lambda y,D_{y}\tilde{V}\left(  s,x,y\right)
\rangle=\langle-D_{x}\tilde{V}\left(  s,x,y\right)  ,\nabla\varphi_{\epsilon
}\left(  x\right)  \rangle,\medskip\\
\text{for }\left(  s,x,y,z\right)  \in\left(  0,T\right)  \times
\overline{\mathrm{Dom}\left(  \varphi\right)  }\times\mathbb{R}^{2d}%
,\medskip\\
\multicolumn{1}{l}{\tilde{V}\left(  T,x,y\right)  =h\left(  x,y\right)  \text{
for }\left(  x,y\right)  \in\overline{\mathrm{Dom}\left(  \varphi\right)
}\times\mathbb{R}^{d}.}%
\end{array}
\right.
\]
Being a viscosity subsolution the following inequality is satisfied:%
\[%
\begin{array}
[c]{r}%
-\dfrac{\partial\Psi}{\partial t}\left(  s_{\epsilon},x_{\epsilon}%
,y_{\epsilon}\right)  +\sup_{u\in\mathrm{U}}\mathcal{H}\left(  s_{\epsilon
},x_{\epsilon},y_{\epsilon},z,u,-D_{x}\Psi\left(  s_{\epsilon},x_{\epsilon
},y_{\epsilon}\right)  ,-D_{xx}^{2}\Psi\left(  s_{\epsilon},x_{\epsilon
},y_{\epsilon}\right)  \right)  \medskip\\
-\langle x_{\epsilon}-e^{-\lambda\delta}z-\lambda y_{\epsilon},D_{y}%
\Psi\left(  s_{\epsilon},x_{\epsilon},y_{\epsilon}\right)  \rangle\leq
\langle-D_{x}\Psi\left(  s_{\epsilon},x_{\epsilon},y_{\epsilon}\right)
,\nabla\varphi_{\epsilon}\left(  x_{\epsilon}\right)  \rangle\,.
\end{array}
\]
Using (\ref{ineq Yosida})-$\left(  iii\right)  $, the convergence
$J_{\epsilon}\left(  x_{\epsilon}\right)  \rightarrow x$ as $\epsilon
\rightarrow0$ and definition (\ref{def for visc}), we obtain%
\[%
\begin{array}
[c]{r}%
-\dfrac{\partial\Psi}{\partial t}\left(  s,x,y\right)  +\sup_{u\in\mathrm{U}%
}\mathcal{H}\left(  s,x,y,z,u,-D_{x}\Psi\left(  s,x,y\right)  ,-D_{xx}^{2}%
\Psi\left(  s,x,y\right)  \right)  \medskip\\
-\langle x-e^{-\lambda\delta}z-\lambda y,D_{y}\Psi\left(  s,x,y\right)
\rangle\leq\partial\varphi^{\ast}(x;-D_{x}\Psi(t,x,y))\,.
\end{array}
\]
Similar arguments show that $v$ is also a viscosity supersolution. Hence $v$
is a viscosity solution of HJB (\ref{HJB}).\hfill
\end{proof}

\end{document}